\documentclass[a4paper,11pt]{article}

\usepackage{authblk}
\usepackage{hyperref}

\title{Estimation of the Hurst and the stability indices\\
 of a $H$-self-similar stable process}
\author{Thi To Nhu DANG\footnote{The university of Danang, University of Economics, 71 Ngu Hanh Son, Danang, Vietnam. E-mail: nhudtt@due.edu.vn},
Jacques ISTAS\footnote{Laboratoire Jean Kuntzmann, Universit\'e de Grenoble Alpes et CNRS, F-38000 Grenoble, France.\\ E-mail: jacques.istas@univ-grenoble-alpes.fr }}
\usepackage[utf8x]{inputenc}

\usepackage{lmodern}
\usepackage{fancyhdr}
\usepackage[top=2cm, bottom=3cm, left=2cm, right=1.5cm]{geometry}
\usepackage{setspace}
\usepackage{textcomp}
\usepackage{graphicx}
\usepackage{subfig}
\usepackage[nottoc, notlof, notlot]{tocbibind}
\usepackage{url}
\usepackage{color}
\usepackage{bbm}% de co the ho tro viet phong chu to cho cac so nhu ham chi 1

\usepackage{amsthm}
\usepackage{amsmath}
\usepackage{amssymb}
\usepackage{mathrsfs}
\usepackage{dsfont}

\usepackage{listings}				% permet l'insertion de code
{\theoremstyle{definition}
\newtheorem{defn}{Definition}[section]}% moi truong dinh nghia

\newtheorem{thm}{Theorem}[section]
\newtheorem{remark}{Remark}[section]
\newtheorem{prop}{Proposition}[section]

\newtheorem{lem}{Lemma}[section]

\begin{document}

\maketitle
\begin{abstract}
In this paper we estimate both the Hurst and the stability indices of a $H$-self-similar stable process.
More precisely, let $X$ be a $H$-sssi (self-similar stationary increments) symmetric $\alpha$-stable process.
The process $X$ is observed at points $\frac{k}{n}$, $k=0,\ldots,n$. 
Our estimate is based on $\beta$-negative power variations with $-\frac{1}{2}<\beta<0$. 
We obtain consistent estimators, with rate of convergence, for several classical $H$-sssi $\alpha$-stable processes 
(fractional Brownian motion, well-balanced linear fractional stable motion, Takenaka's process, L\'evy motion).
Moreover, we obtain asymptotic normality of our estimators for fractional Brownian motion and L\'evy motion.
\end{abstract}
\tableofcontents
\section{Introduction}

Self-similar processes play an important role in probability because of their connection to limit theorems 
and they are widely used to model natural phenomena. 
For instance, persistent phenomena in internet traffic, hydrology, geophysics or financial markets, e.g., \cite{Cohen2013}, \cite{McCulloch1996}, \cite{Taqqu1994}, are known to be self-similar.  
Stable processes have attracted growing interest in recent years: 
data with ''heavy tails'' have been collected in fields as diverse as economics, telecommunications, hydrology and physics of
condensed matter, which suggests using non-Gaussian stable processes as possible models, e.g., \cite{Taqqu1994}.
Self-similar $\alpha$-stable processes have been proposed to model some natural phenomena with heavy tails, as in \cite{Taqqu1994} and references therein. 

The estimation of various indices of $H-$sssi $\alpha-$stable processes has been a problem studied since several decades ago and, even nowadays, it continues to be a challenge. 
%In the past two decades a lot of work has been reported in the literature statistics on estimating various indices of $H-$sssi $\alpha-$stable processes. 
In the case of fractional Brownian motion, the estimation of the self-similarity index $H$ has attracted attention to many authors and many methods have been proposed for solving this problem. Among these, one can mention the quadratic variation method (see e.g. \cite{Benassi2000}, \cite{Benassi1998}, \cite{Cohen2013}, \cite{Jacques1997}), the $p$-variation method (see e.g. \cite{Coeurjolly2001}, \cite{Nourdin2010}), the wavelet coefficients method (see e.g. \cite{Delbeke1999}, \cite{Bardet2010}, \cite{Lacaux2007}), the log-variation method (see e.g. \cite{Cohen2013}, \cite{Jacques2012ESAIM}). %, the fractional dimension method. %For linear fractional stable motions, in , the authors have proposed strongly consistent estimators of the parameter $H$, based on the discrete wavelet transform of the processes. 
Other references, like the works of J. Istas, recommend the use of complex variations for estimating the self-similarity index $H$ of $H-$sssi processes, but not for estimating $\alpha$, (see e.g. \cite{Jacques2012}). For linear fractional stable motions, strongly consistent estimators of the self-similarity index $H$, based on the discrete wavelet transform of the processes, have been proposed without requirement that $\alpha$ to be known, as in \cite{Taqqu1999},  \cite{Pipiras2007}, \cite{Taqqu2002}, \cite{Taqqu2005}. Thus, regarding the estimation of the stability index $\alpha$, in \cite{Ayache2013}, the authors presented a wavelet estimator for linear fractional stable motions assuming that $H$ is known. Recently, the corresponding estimation problem of the stability function and the localisability function for a class of multistable processes was considered in the discussion paper of R. Le Gu\'evel, see \cite{Guevel2013}, based on some conditions that involve the consistency of the estimators. For linear multifractional stable motions, in \cite{Ayache2015}, the authors presented strongly consistent estimators of the localisability function $H(.)$ and the stability index $\alpha$ using wavelet coefficients when $\alpha\in(1,2)$ and $H(.)$ is a H\"older function smooth enough,
with values in a compact subinterval $[\underline{H}, \overline{H}]$ of $(1/\alpha, 1)$. % for the existence of the underlying processes.  

The aim of this work is to construct consistent estimators of the self-similar index $H$ and the stable index $\alpha$ of $H$-sssi, $S\alpha S$-stable processes using a new framework. In the view of the fact that a stable random variable has a density function, $\beta-$ negative power variations have expectations and covariances for $-1/2<\beta<0$. Our estimates are thus based on these variations. This new approach provides estimators of $H$ and $\alpha$ without assumptions on the existence moments of the underlying processes. It also allows us to give an estimator for the self-similarity parameter $H$ without assumption on $\alpha$ and vice versa, we can estimate the stability index $\alpha$ without assumption on $H$. In other words, using $\beta-$ negative power variations ($-1/2<\beta<0$), one can obtain the estimators of $H$ and $\alpha$ separately. 
%In this work we are interested in $H$-sssi, $S\alpha S$-stable processes. 
We prove the consistency and rates of convergence of the proposed estimators for $H$ and $\alpha$ for the underlying processes under an assumption on the series of covariances of $\beta$-negative power variations ($-1/2< \beta<0$). Then obtained results were illustrated by some classical examples: fractional Brownian motions, $S\alpha S$-stable Lévy motions, well-balanced linear fractional stable motions and Takenaka’s processes. We then show that the asymptotic normality of our estimates can be  ascertained for the proposed estimators when the underlying process is a fractional Brownian motion or an $S\alpha S$-stable Lévy motion.

The remainder part of this article is organized as follows: in the next section, we present the setting, the assumption and main results %\ref{section.2} 
to construct the estimators of $H$ and $\alpha$. %along with the main results regarding the convergence of these estimators on four examples (fractional Brownian motions,
%$S\alpha S-$stable L\'evy motions,
%well-balanced linear fractional stable motions, Takenaka's processes) and the central limit theorem in the cases of  the fractional Brownian motion and the $S\alpha S-$stable L\'evy motion.
 %Section \ref{section.3} is to present results on negative power variations which are used for the proofs of results in illustrated examples in latter section.
 % Section \ref{section.4} gives the proofs for main results presented in Section \ref{section.2}.
In Section \ref{section.3}, some classical examples for the obtained results in Section \ref{section.2} are given: fractional Brownian motions,
$S\alpha S-$stable L\'evy motions,
well-balanced linear fractional stable motions, Takenaka's processes. In this Section, we also show the central limit theorem for the cases of  the fractional Brownian motion and the $S\alpha S-$stable L\'evy motion. 
Finally, in Section \ref{section.4}, we gather all the proofs of the main results and of the illustrated examples: Subsection \ref{subsection.4.1} contains auxiliary results on negative power variations which play an important role in the proofs in Subsection \ref{subsection.4.2} of the main results and in the proofs in Subsection \ref{subsection.4.3} of the results of four examples.
%for the proofs in later subsections: the main results in Subsection \ref{subsection.4.2}, four examples in Subsection \ref{subsection.4.3}. 
%Section \ref{section.4} gives the proofs for main results presented in Section \ref{section.2}.
%Finally, in Section \ref{section.6}, we gather all the proofs of auxiliary results which are used in the proofs of the latter sections.
%Finally, we gather all the proofs of the main results in Section \ref{section.6}.
\section{Main results}\label{section.2}
%\subsection{Setting and main results}
Let us recall the definition of a $H-$sssi process and an $\alpha-$ stable process (see e.g., \cite{Taqqu1994}): 
A real-valued process $X$
\begin{itemize} %A real-valued process $X$
\item is {\it{H-self-similar (H-ss)}} if for all $a>0$,
$
\{X(at), t\in \mathbb{R}\}\stackrel{(d)}{=}a^H\{X(t), t\in \mathbb{R}\},
$
\item has {\it {stationary increments (si)}} if, for all $s\in \mathbb{R}$,
$
\{X(t+s)-X(s), t\in \mathbb{R}\}\stackrel{(d)}{=}\{X(t)-X(0), t\in \mathbb{R}\}
$
\end{itemize}
where $\stackrel{(d)}{=}$ stands for equality of finite dimensional distributions.
A random variable $X$ is said to have a symmetric {\it{$\alpha$-stable distribution}} ($S\alpha S$) if there are parameters $\alpha\in (0,2]$ and $\sigma>0$ such that its characteristic function has the form:
 $$ \mathbb{E} e^{i\theta X}=
   \exp\left(-\sigma^\alpha \mid \theta\mid^\alpha\right).
  $$
%\item
 When $\sigma=1$, a $S\alpha S$ is said to be standard. Let $X$ be a $H$-sssi, $S\alpha S$ random process with $0<\alpha\leq 2$.\\
Let $L\geq 1, K\geq 1$ be fixed integers, $a=(a_0,\ldots,a_K)$ be a finite sequence with exactly $L$ vanishing first moments, that is for all $q\in   \{0,\ldots, L\}$, one has
\begin{align}\label{eq.A}
\sum\limits_{k=0}^K k^qa_k&=0,
\sum\limits_{k=0}^K k^{L+1}a_k\neq 0
\end{align}
with convention $0^0=1$. For example, here we can choose $K=L+1$ and 
\begin{align}\label{eq.B}
a_k=(-1)^{L+1-k}\frac{(L+1)!}{k!(L+1-k)!}.
\end{align}
 %For example with $K=2$
The increments of $X$ with respect to the sequence $a$ are defined by
\begin{equation}\label{eq.8.0}
\triangle_{p,n}X=\sum\limits_{k=0}^Ka_kX(\frac{k+p}{n}).
\end{equation}
We define now an estimator of $H$. Let $\beta\in \mathbb{R},-\frac{1}{2}<\beta<0$, we set 
% and $\beta_1, \beta_2,u,v$ be in $\mathbb{R}$ such that $-1/2<\beta_1<\beta_2<0$.
\begin{align}
 V_n({\beta})&=\frac{1}{n-K+1}\sum\limits_{p=0}^{n-K}|\triangle_{p,n}X|^{\beta},\label{eq.8.1}\\
 W_n(\beta)&=n^{\beta H}V_n(\beta). \label{eq.8.2}
 \end{align}
 Notice that $V_n(\beta)$ is the empirical mean of order $\beta$ and $W_n(\beta)$ is expected to converge to its mean. % We define 
 %\begin{align}
%W_n(\beta)&=n^{\beta H}V_n(\beta). \label{eq.8.2}
%\end{align}
 The estimator of $H$ is defined by 
% $\widehat{H}_n, \widehat{\alpha}_n$ be defined by
\begin{align}
\widehat{H}_n&=\frac{1}{\beta}\cdot \log_2\frac{V_{n/2}(\beta)}{V_n(\beta)}. \label{eq.8}
%\widehat{\alpha}_n&=\varphi_{-\beta_1,-\beta_2}
%\left(\psi_{-\beta_1,-\beta_2}
%(W_n(\beta_1),W_n(\beta_2))\right) \label{eq.21}
\end{align}
%where $\varphi_{-\beta_1,-\beta_2},\psi_{-\beta_1,-\beta_2} $ are defined as follows.\\
We are now in position to define an estimator of $\alpha$. We define first auxiliary functions $\psi_{u,v},h_{u,v},\varphi_{u,v}$ before introducing the estimator of $ \alpha$, where $u>v>0$.\\  
%For $0< v<u $, 
% let $g_{u,v}: (0,+\infty)\rightarrow \mathbb{R}$ be the function defined by
 %\begin{equation}\label{eq.17.1}
  %  g_{u,v}(x)=u\ln \left(\Gamma(1+vx)\right)- 
 %v\ln \left(\Gamma(1+ux)\right).
%\end{equation}
%where $u,v$ are parameters such that $u,v\in \mathbb{R}, 0< v<u$.\\
Let $\psi_{u,v}$: $\mathbb{R^+}\times\mathbb{R^+}\rightarrow \mathbb{R}$ be the function defined by
\begin{equation}\label{eq.19.1}
\psi_{u,v}(x,y)=-v\ln x+u\ln y+C(u,v),
\end{equation}
where $C(u,v)=\frac{u-v}{2}\ln (\pi)+u\ln \left( \Gamma(1+\frac{v}{2})\right)+v \ln \left(\Gamma(\frac{1-u}{2})\right)
 -v \ln \left(\Gamma(1+\frac{u}{2})\right)
 -u\ln \left(\Gamma(\frac{1-v}{2})\right)$.\\
Let $h_{u,v}: (0,+\infty)\rightarrow (-\infty,0)$ be the function defined by
\begin{equation}\label{eq.18.1}
h_{u,v}(x)=u\ln \left(\Gamma(1+\frac{v}{x})\right)- 
 v\ln \left(\Gamma(1+\frac{u}{x})\right)
%g_{u,v}(\frac{1}{x}).
\end{equation}
We will prove later that $h_{u,v}$ is bijective.
%Let $\psi_{u,v}$ be the function: $\mathbb{R^+}\times\mathbb{R^+}\rightarrow \mathbb{R}$ defined by
%\begin{equation}\label{eq.19.1}
%\psi_{u,v}(x,y)=-v\ln x+u\ln y+C(u,v),
%\end{equation}
%where $C(u,v)=\frac{u-v}{2}\ln (\pi)+u\ln \left( \Gamma(1+\frac{v}{2})\right)+v \ln \left(\Gamma(\frac{1-u}{2})\right)
 %-v \ln \left(\Gamma(1+\frac{u}{2})\right)
 %-u\ln \left(\Gamma(\frac{1-v}{2})\right)$.\\
 Let $\varphi_{u,v}:\mathbb{R}\rightarrow [0,+\infty)$ be the function defined by
  \begin{equation}\label{eq.20.1}
  \varphi_{u,v}(x)=\begin{cases}
          0 &\mbox {  if $x\geq 0$}\\
          h^{-1}_{u,v}(x)& \mbox {   if $x<0$}
         \end{cases}
\end{equation}
  where $h_{u,v}$ is defined as in (\ref{eq.18.1}).\\
  Let $\beta_1,\beta_2$ be in $\mathbb{R}$ such that $-1/2<\beta_1<\beta_2<0$. The estimator of $ \alpha$ is defined by
%Let $\hat{\alpha}_n$ be defined by
\begin{equation}\label{eq.21}
\hat{\alpha}_n=\varphi_{-\beta_1,-\beta_2}\left(\psi_{-\beta_1,-\beta_2}(W_n(\beta_1),W_n(\beta_2))\right),
\end{equation}
where $\psi_{u,v}, \varphi_{u,v}$ are defined as in (\ref{eq.19.1}) and (\ref{eq.20.1}), respectively.\\
With $\beta \in (-\frac{1}{2},0)$ fixed, we will make the following assumption:
%\begin{enumerate}
%\item
%\begin{equation}\label{eq.9.1}
 %\lim_{n\rightarrow +\infty}\frac{1}{n}\sum_{k\in \mathbb{Z},|k|\leq n}
 %\mid cov(\mid\Delta_{k,1}X\mid^\beta,\mid\Delta_{0,1}X\mid^\beta)\mid=0,
  % \end{equation}
 %\item  
There exist a sequence $\{b_n,n\in\mathbb{N}\}$ and a constant $C$ such that $\lim\limits_{n\rightarrow +\infty}b_n=0, b_{n/2}=O(b_n)$ and
\begin{equation}\label{eq.9.2}
 \limsup_{n\rightarrow +\infty}\frac{1}{nb_n^2}\sum_{k\in \mathbb{Z},|k|\leq n}
 | cov(|\triangle_{k,1}X|^\beta,|\triangle_{0,1}X|^\beta)|\leq C^2.
  \end{equation}   
 %\end{enumerate}  
\begin{remark}\label{remark.2}
The assumption (\ref{eq.9.2}) is important to prove the consistency of the estimators of the self-similarity and the stability indices. We will see its role in the main theorem below.
\end{remark}
Now we are in position to present our main results for the estimation of $H$ and $\alpha$, based on the assumption (\ref{eq.9.2}).
%Let $X$ be a $H$-sssi, $S\alpha S$ random process, 
%$\beta,\beta_1,\beta_2\in\mathbb{R}, -\frac{1}{2}<\beta<0, -\frac{1}{2}<\beta_1<\beta_2<0$.
%We have the following results for the estimation of $H$ and $\alpha$, based on the assumption (\ref{eq.9.2}), 
%where $O_{\mathbb{P}}$ is defined by:\\
%$\bullet X_n=O_{\mathbb {P}}(1)$ iff for all $\epsilon>0$, 
%there exists $M>0$ such that $\sup\limits_{n}\mathbb{P}(|X_n|>M)<\epsilon$.\\
%$\bullet Y_n=O_{\mathbb {P}}(a_n)$ means $Y_n=a_nX_n$ %with $X_n=O_{\mathbb {P}}(1)$.
\begin{thm}\label{thm.5}
Let $X$ be a $H$-sssi, $S\alpha S$ random process that satisfies assumption (\ref{eq.9.2}). Also, let $\beta,\beta_1,\beta_2\in\mathbb{R}, -\frac{1}{2}<\beta<0, -\frac{1}{2}<\beta_1<\beta_2<0$ and $\widehat{H}_n,\hat{\alpha}_n$ be defined as in (\ref{eq.8}) and (\ref{eq.21}), respectively. %1. Assuming (\ref{eq.9.1}), then as $n\rightarrow  +\infty$\\
 %  $$\widehat{H}_n\xrightarrow{\mathbb{P}} H, \hat{\alpha}_n\xrightarrow{\mathbb{P}} \alpha.$$
  %where $ \widehat{H}_n $ is defined as in (\ref{eq.8}).\\
 %Assuming (\ref{eq.9.2}),
Then as $n\rightarrow  +\infty$, one has 
 $$\widehat{H}_n\xrightarrow{\mathbb{P}} H,
\hat{\alpha}_n\xrightarrow{\mathbb{P}} \alpha, $$ moreover
 $ \widehat{H}_n-H=O_{\mathbb{P}}(b_n),
  \hat{\alpha}_n-\alpha=O_{\mathbb{P}}(b_n)$, where $O_{\mathbb{P}}$ is defined by:\\
$\bullet X_n=O_{\mathbb {P}}(1)$ iff for all $\epsilon>0$, 
there exists $M>0$ such that $\sup\limits_{n}\mathbb{P}(|X_n|>M)<\epsilon$,\\
$\bullet Y_n=O_{\mathbb {P}}(a_n)$ means $Y_n=a_nX_n$ with $X_n=O_{\mathbb {P}}(1)$.
  \end{thm}
  See Subsection \ref{subsection.4.2} for the proof of Theorem \ref{thm.5}.
  \section{Examples}\label{section.3}
In this section, we study % state main results giving the rate of convergence in Section \ref{section.2} for 
 four classical examples: 
fractional Brownian motion, $S\alpha S$-stable L\'evy motion,
well-balanced linear fractional stable motion, Takenaka's process. %, with a central limit theorem in the first two cases. 
For these, we will show in Section \ref{section.4} that (\ref{eq.9.2}) is valid, so that the conclusion of Theorem \ref{thm.5} holds. We precise this theorem by providing the rate of convergence defined in (\ref{eq.9.2}) and a central limit theorem for the first two cases.  % , the assumption (\ref{eq.9.2}) is satisfied, then we obtain the results in Theorem \ref{thm.5}.  
\subsection{Fractional Brownian motion}\label{subsection.3.1}
\begin{defn}\emph{Fractional Brownian motion}\\
Fractional Brownian motion is a centered Gaussian process with covariance given by
$$\mathbb{E}X(t)X(s)=\frac{\mathbb{E}X(1)^2}{2}\{|s|^{2H}+|t|^{2H}-|s-t|^{2H}\}.$$
\end{defn}
Fractional Brownian motion is a $H$-sssi 2-stable process (see, e.g., \cite{Cohen2013}, p. 59). We will prove that the condition (\ref{eq.9.2}) is satisfied with $b_n=n^{-1/2}$, then the results in Theorem \ref{thm.5} are obtained. Moreover, we can obtain the asymptotic normality of the estimators of the self-similarity index $H$ and the stability index $\alpha=2$.\\
 Let $X$ be a $H$ fractional Brownian motion with $H\in(0,1)$. We first present the variances $\Xi_1, \Sigma_1$ for the limit distributions of the central limit theorems for the estimators of $H$ and $\alpha$.\\ 
We will mimic the Breuer-Major's theorem (see e.g., Theorem 7.2.4 in \cite{Nourdin2012}) to define these variances. %Let us recall this theorem (see e.g., Theorem 7.2.4 in \cite{Nourdin2012}). The Hermite rank $d$ of a function $f$ is the order of the first non-trivial coefficient in the Hermite expansion of $f$.  
  %Let $X$ be a $H$ fractional Brownian motion with $H\in(0,1)$.\\ %Now we will present $\Xi_1$ and $\Sigma_1$.\\
 %For $\beta\in\mathbb{R},-1/2<\beta<0$, 
% For $k\in \mathbb{N}$, set
 %\begin{equation}\label{eq.19.0}
%Z_k=\frac{\triangle_{k,1}X}{\sqrt{var \triangle_{k,1}X}}.
%\end{equation} 
%where $\triangle_{k,1}X$ is defined by (\ref{eq.8.0}). 
 %It is clear that $Z_k\sim \mathcal{N}(0,1)$,
%$\{Z_k\}_{k\geq 0}$ is a centered stationary Gaussian family. % and for $k,l\geq 0$, 
%\begin{align}
% \mathbb{E}Z_kZ_l&=\frac{\sum\limits_{p,p'=0}^{K}a_pa_{p'}|k-l+p-p'|^{2H}}
 %{\sum\limits_{p,p'=0}^{K}a_pa_{p'}|p-p'|^{2H}}\notag \\
% &=\rho(k-l), \label{eq.19.2}
%\end{align}
%where $\rho$ is defined by (\ref{eq.11.2.1}). %Here $\rho(0)=1$.\\
For $\beta\in\mathbb{R},-1/2<\beta<0$, let us introduce the following function
\begin{align}\label{eq.11.0}
 f_{\beta}(x)=\sqrt{var \triangle_{0,1}X}^{\beta}(|x|^\beta-\mathbb{E}|Z_0|^\beta),
\end{align}
where $Z_0=\frac{\triangle_{0,1}X}{\sqrt{var \triangle_{0,1}X}}$.\\
%We have 
%$\frac{1}{\sqrt{2\pi}}\int\limits_{\mathbb{R}}f(x)e^{-x^2/2}dx=0$ and 
%$
%\mathbb{E}f^2(Z_0)= \frac{1}{\sqrt{2\pi}}\int\limits_{\mathbb{R}}f^2(x)e^{-x^2/2}dx<+\infty
%$
%since $-1/2<\beta<0$. Then %from Property \ref{property.A.1}
Following Proposition \ref{prop.A.1} in Appendix,  we can write $f_\beta$ in terms of Hermite polynomials in a unique way 
\begin{align}\label{eq.11.1}
f_\beta(x)=\sum_{q \geq d}f_{\beta,q}H_q(x),
\end{align}
where $d$ is the Hermite rank of $f_{\beta}$ and $d\geq 2, \sum\limits_{q\geq d}q!f^2_{\beta,q}<+\infty $.  
 %From Proposition \ref{prop.A.4.2},
 % However, we have $\mathbb{E}H_0(Z_0)f(Z_0)=\mathbb{E}f(Z_0)=0$
 %and
 %$$
 %\mathbb{E}H_1(Z_0)f(Z_0)=
 %\mathbb{E}Z_0f(Z_0)=\frac{(var \triangle_{0,1}X)^{\beta}}{\sqrt{2\pi}}\int\limits_{\mathbb{R}}
 % x(|x|^\beta-\mathbb{E}|Z_0|^\beta)e^{-x^2/2}dx=0.
 %$$
%On the other hand, $\mathbb{E}H_0(Z_0)f(Z_0)=f_0, 
 %\mathbb{E}H_1(Z_0)f(Z_0)=f_1\mathbb{E}Z^2_0=f_1,
 %$
 %thus $f_1=0$ and 
 %then it follows that $ f_1=f_0=0, d\geq 2$. 
 %Since $Z_0\sim \mathcal{N}_1(0,1)$ and $-1/2<\beta<0$, we deduce
 %\begin{align*}
  %\mathbb{E}f^2(Z_0)&=\frac{(var \triangle_{0,1}X)^{\beta}}{\sqrt{2\pi}}\int\limits_{\mathbb{R}}
  %(|x|^\beta-\mathbb{E}|Z_0|^\beta)^2e^{-x^2/2}dx <+\infty.
 %\end{align*}
  %Then 
%\begin{align}\label{eq.11.2}
%\mathbb{E}f^2(Z_0)=\sum\limits_{q\geq d}q!f^2_q<+\infty.
%\end{align}
Let
\begin{equation}\label{eq.11.2.1}
 \rho(r)=\frac{\sum\limits_{p,p'=0}^Ka_pa_{p'}|r+p-p'|^{2H}}{\sum\limits_{p,p'=0}^Ka_pa_{p'}|p-p'|^{2H}},
 \end{equation}
 \begin{equation}\label{eq.11.2.2}
 \rho_1(r)
 =\frac{\sum\limits_{p,p'=0}^Ka_pa_{p'}|r+p-2p'|^{2H}}{2^H\sum\limits_{p,p'=0}^Ka_pa_{p'}|p-p'|^{2H}},
\end{equation}
\begin{equation}\label{eq.11.2.3}
 \varGamma_1=\begin{pmatrix}
              \sum\limits_{q\geq d}q!f_{\beta,q}^2\sum\limits_{r\in \mathbb{Z}}\rho^q(r)&
              \sum\limits_{q\geq d}q!f_{\beta,q}^2\sum\limits_{r\in \mathbb{Z}}\rho_1^q(r)\\
              &\\
              \sum\limits_{q\geq d}q!f_{\beta,q}^2\sum\limits_{r\in \mathbb{Z}}\rho_1^q(r)&
              2\sum\limits_{q\geq d}q!f_{\beta,q}^2\sum\limits_{r\in \mathbb{Z}}\rho^q(r).\\
             \end{pmatrix}
\end{equation}
and  $\phi: \mathbb{R}^+\times \mathbb{R}^+\rightarrow  \mathbb{R}$
be defined by
\begin{equation}\label{eq.11.3.1}
\phi(x,y) =\frac{1}{\beta}\log_2\frac{x}{y}.
\end{equation}
Then $\Xi_1$ is defined by 
\begin{equation}\label{eq.11.2.4}
 \Xi_1=\phi'(x_0,y_0)\varGamma_1\phi'(x_0,y_0)^t,
\end{equation}
where 
\begin{equation}\label{eq.11.2.5}
(x_0,y_0)=(\mathbb{E}|\triangle_{0,1}X|^\beta,\mathbb{E}|\triangle_{0,1}X|^\beta).
\end{equation}
To define $\Sigma_1$, let $-1/2<\beta_1<\beta_2<0$, following Proposition \ref{prop.A.1} in Appendix,  we can write $f_{\beta_1}, f_{\beta_2}$ in terms of Hermite polynomials in a unique way 
\begin{align}\label{eq.21.1}
f_{\beta_1}(x)=\sum_{q \geq d_1}f_{\beta_1,q}H_q(x),f_{\beta_2}(x)=\sum_{q \geq d_1}f_{\beta_2,q}H_q(x)
\end{align}
where $d_1$ is the minimum of the Hermite ranks of $f_{\beta_1}$ and $f_{\beta_2}$,  $d_1\geq 2$ and 
$$ \sum\limits_{q\geq d}q!f^2_{\beta_1,q}<+\infty, \sum\limits_{q\geq d}q!f^2_{\beta_2,q}<+\infty .$$ 
% Set  
%\begin{equation}\label{eq.19.0}
%Z_k=\frac{\triangle_{k,1}X}{\sqrt{var \triangle_{k,1}X}}.
%\end{equation} 
%It is clear that $Z_k\sim \mathcal{N}(0,1)$,
%$\{Z_k\}_{k\geq 0}$ is a centered stationary Gaussian family and for $k,l\geq 0$, 
%\begin{align}
 %\mathbb{E}Z_kZ_l&=\frac{\sum\limits_{p,p'=0}^{K}a_pa_{p'}|k-l+p-p'|^{2H}}
 %{\sum\limits_{p,p'=0}^{K}a_pa_{p'}|p-p'|^{2H}}\notag \\
 %&=\rho(k-l), \label{eq.19.2}
%\end{align}
%where $\rho$ is defined by (\ref{eq.11.2.1}). %Here $\rho(0)=1$.\\
%Let
%\begin{align}
%h(x)&=|\sqrt{var\triangle_{0,1}X}|^{\beta_1}\left(|x|^{\beta_1}
%-\mathbb{E}|Z_0|^{\beta_1}\right),\notag\\ 
%g(x)&=|\sqrt{var\triangle_{0,1}X}|^{\beta_2}\left(|x|^{\beta_2}
%-\mathbb{E}|Z_0|^{\beta_2}\right).\label{eq.20.0}
%\end{align}
%It is obvious that $$\mathbb{E}h^2(Z_0)=\frac{1}{\sqrt{2\pi}}
% \int_{\mathbb{R}} (var\triangle_{0,1}X)^{\beta_1}\left(|x|^{\beta_1}
%-\mathbb{E}|Z_0|^{\beta_1}\right)^2e^{-x^2/2}dx<+\infty,$$
%$$\mathbb{E}g^2(Z_0)=\frac{1}{\sqrt{2\pi}}
 %\int_{\mathbb{R}}(var\triangle_{0,1}X)^{\beta_2}\left(|x|^{\beta_2}
%-\mathbb{E}|Z_0|^{\beta_2}\right)^2e^{-x^2/2}dx<+\infty.$$
%Following Proposition 1.4.2-(iv) in \cite{Nourdin2012}, we can write $h, g$ in terms of Hermite polynomials in a unique way:
%\begin{equation}\label{eq.21.1}
 %h(x)=\sum_{q=d_1}^{+\infty}h_qH_q(x),g(x)=\sum_{q=d_1}^{+\infty}g_qH_q(x)
%\end{equation}
%where $h_q,g_q\in \mathbb{R}$ and $d_1$ is the minimum of the Hermite ranks of $h$ and $g$.\\
%Since $\mathbb{E}h(Z_0)=\mathbb{E}g(Z_0)=0$, it follows that $d_1\geq 1$.
%In a similar way to $d$, we can obtain that $d_1\geq 2$. 
 Let
\begin{equation}\label{eq.22}
\Sigma_1=\nabla_{\varphi_{-\beta_1,-\beta_2} \circ \psi_{-\beta_1,-\beta_2}}(x_1,y_1) \varGamma_2 \nabla_{\varphi_{-\beta_1,-\beta_2} \circ \psi_{-\beta_1,-\beta_2}}(x_1,y_1)^t
\end{equation}
where  $\psi_{u,v}, \varphi_{u,v}$ are defined by (\ref{eq.19.1}), (\ref{eq.20.1}) respectively, $\nabla$ is the differential operator and 
\begin{align}
(x_1,y_1)&=(\mathbb{E}|\triangle_{0,1}X|^{\beta_1},\mathbb{E}|\triangle_{0,1}X|^{\beta_2}),\label{eq.22.1}\\
\varGamma_2&=\begin{pmatrix}
           \sigma_{\beta_1}^2 & \rho_{\beta_1,\beta_2}\\
           \rho_{\beta_1,\beta_2} & \sigma_{\beta_2}^2
          \end{pmatrix}, \label{eq.23}
\end{align}
\begin{align}
\sigma_{\beta_1}^2&=\sum_{q=d_1}^{+\infty} q! f_{\beta_1,q}^2\sum_{k\in \mathbb{Z}}\rho(k)^q,
\sigma_{\beta_2}^2=\sum_{q=d_1}^{+\infty} q! f_{\beta_2,q}^2\sum_{k\in \mathbb{Z}}\rho(k)^q,
\rho_{\beta_1,\beta_2}=\sum_{q=d_1}^{+\infty} q! f_{\beta_1,q}f_{\beta_2,q}\sum_{k\in \mathbb{Z}}\rho(k)^q\label{eq.26}.
\end{align}
We can now state the following theorem, which precises the results for the estimation of $H$ and $\alpha$ in the case of fractional Brownian motion.
\begin{thm}\label{thm.6}
Let $X$ be a fractional Brownian motion. Then\\
a) $$\widehat{H}_n- H= O_\mathbb{P}(n^{-1/2}), \hat{\alpha}_n-2=O_\mathbb{P}(n^{-1/2}),$$
b) $$\sqrt{n}(\widehat{H}_{n}-H)
\xrightarrow{(d)}\mathcal{N}_1(0,\Xi_1)),\sqrt{n}(\hat{\alpha}_{n}-2)
\xrightarrow{(d)}\mathcal{N}_1(0,\Sigma_1) $$
as $n\rightarrow +\infty$, 
where $\Xi_1, \Sigma_1$ are defined by (\ref{eq.11.2.4}) and (\ref{eq.22}), respectively.
\end{thm}
See Subsection \ref{subsection.4.3} for the proof of Theorem \ref{thm.6}.
\subsection{\texorpdfstring{$S\alpha S$-}{}stable L\'evy motion}
\begin{defn}\emph{$S\alpha S$-stable L\'evy motion}\\
A stochastic process $\{X(t), t\geq 0\}$ is called (standard) $S\alpha S$-stable L\'evy motion if
%\begin{enumerate}
%\item 
$X(0)=0$ (a.s.), %\item
$X$ has independent increments and, %\item
%$X(t)-X(s)\sim S_{\alpha}((t-s)^{1/\alpha},\beta,0)$ 
for all $0\leq s<t<\infty $ and 
for some $0<\alpha\leq 2$, $X(t)-X(s)$ is a $S\alpha S$ random variable with characteristic function given by
$$ \mathbb{E} e^{i\theta (X(t)-X(s))}=
   \exp\left(-(t-s) | \theta|^\alpha\right).
  $$
  \end{defn}
The condition (\ref{eq.9.2}) is proved to be  satisfied with $b_n=n^{-1/2}$, then the results in Theorem \ref{thm.5} are ascertained. Similar to the case of fractional Brownian motion, we obtain the asymptotic normality of $H$ and $\alpha$.\\
The variances $\Xi_2, \Sigma_2$ for the limit distributions of the central limit theorems for the estimators of $H$ and $\alpha$ are defined as follows.\\
Let $X$ be a $S\alpha S-$stable L\'evy motion, we define the variance for the limit distribution of the central limit theorem for the estimator of $H$ by
\begin{equation}\label{eq.35}
\Xi_2=\phi'(x_0,y_0)\varGamma_3 \phi'(x_0,y_0)^t,
\end{equation}
where $\phi(x,y), (x_0,y_0)$ are defined by (\ref{eq.11.3.1}), (\ref{eq.11.2.5}), respectively and
\begin{align}
\varGamma_3&=\begin{pmatrix}
\sigma_1^2 & \sigma_{1,2}\\
\sigma_{1,2}&\sigma_2^2
\end{pmatrix},\label{eq.34}
\end{align}
\begin{align}
\sigma_1^2&=var |\triangle_{0,1}X|^\beta+ cov(|\triangle_{0,1}X|^\beta,|\triangle_{1,1}X|^\beta)\notag\\
&+2\sum\limits_{p=1}^{K-1} \left( cov(|\triangle_{0,1}X|^\beta,
|\triangle_{2p,1}X|^\beta)+cov(|\triangle_{0,1}X|^\beta,
|\triangle_{2p+1,1}X|^\beta)\right) \notag\\
&+2\sum\limits_{p=1}^{K-1} \left(cov(|\triangle_{1,1}X|^\beta,
|\triangle_{2p,1}X|^\beta)+cov(|\triangle_{1,1}X|^\beta,
|\triangle_{2p+1,1}X|^\beta)\right),\label{eq.31}\\
 \sigma_2^2&=2\left(var |\triangle_{0,1}X|^\beta+2\sum\limits_{p=1}^{K-1}cov(|\triangle_{0,1}X|^\beta,
|\triangle_{p,1}X|^\beta)\right) \label{eq.32}\\
 \sigma_{1,2}&=2^{\beta H}\left(cov(|\triangle_{0,2}X|^\beta,
|\triangle_{0,1}X|^\beta)+cov(|\triangle_{1,2}X|^\beta,
|\triangle_{0,1}X|^\beta)\right) \notag\\
&+2^{\beta H}\sum\limits_{p=1}^{K-1}\left(cov(|\triangle_{0,2}X|^\beta,
|\triangle_{p,1}X|^\beta)+cov(|\triangle_{1,2}X|^\beta,
|\triangle_{p,1}X|^\beta)\right)\notag\\
&+ 2^{\beta H}\sum\limits_{p=1}^{K-1}\left(cov(|\triangle_{0,1}X|^\beta,
|\triangle_{2p,2}X|^\beta)+cov(|\triangle_{0,1}X|^\beta,
|\triangle_{2p+1,2}X|^\beta)\right)\label{eq.33}
\end{align}
The variance for the limit distribution of the central limit theorem for the estimator of $\alpha$ is defined by
\begin{equation}\label{eq.29}
\Sigma_2=\nabla_{\varphi_{-\beta_1,-\beta_2} \circ \psi_{-\beta_1,-\beta_2}}(x_1,y_1) \varGamma_4 \nabla_{\varphi_{-\beta_1,-\beta_2} \circ \psi_{-\beta_1,-\beta_2}}(x_1,y_1)^t
\end{equation}
where $\psi_{u,v}, \varphi_{u,v}, (x_1,y_1)$ are defined as in (\ref{eq.19.1}), (\ref{eq.20.1}) and (\ref{eq.22.1}), respectively,
\begin{align}
\varGamma_4&=\begin{pmatrix}
\sigma_1^2 & \sigma_{1,2}\\
\sigma_{1,2}&\sigma_2^2
\end{pmatrix},\label{eq.30}
\end{align}
\begin{align}
\sigma_1^2&=var|\triangle_{0,1}X|^{\beta_1}+2\sum\limits_{k=1}^{K-1}cov(|\triangle_{0,1}X|^{\beta_1},|\triangle_{k,1}X|^{\beta_1}),\label{eq.27.0}\\
\sigma_2^2&=var|\triangle_{0,1}X|^{\beta_2}+2\sum\limits_{k=1}^{K-1}cov(|\triangle_{0,1}X|^{\beta_2},|\triangle_{k,1}X|^{\beta_2}) \label{eq.27}\\
\sigma_{1,2}&=cov(|\triangle_{0,1}X|^{\beta_1},|\triangle_{0,1}X|^{\beta_2})+\frac{1}{2}\sum\limits_{k=1}^{K-1} \left( cov(|\triangle_{0,1}X|^{\beta_1},|\triangle_{k,1}X|^{\beta_2})+cov(|\triangle_{0,1}X|^{\beta_2},|\triangle_{k,1}X|^{\beta_1}) \right) \label{eq.28.1}.
\end{align}
We now present the results on the asymptotic normality for the case of $S\alpha S$-stable L\'evy motion.
\begin{thm}\label{thm.7}
Let $X$ be a $S\alpha S$-stable L\'evy motion. Then\\ %as $n\rightarrow  +\infty$\\
%Let $X$ be a $H$ fractional Brownian motion with $H\in(0,1)$. Then
%\begin{enumerate}
%\item 
a) $$\widehat{H}_n- H= O_\mathbb{P}(n^{-1/2}), \hat{\alpha}_n-\alpha=O_\mathbb{P}(n^{-1/2})$$
b) $$\sqrt{n}(\widehat{H}_{n}-H)
\xrightarrow{(d)}\mathcal{N}_1(0,\Xi_2)),\sqrt{n}(\hat{\alpha}_{n}-\alpha)
\xrightarrow{(d)}\mathcal{N}_1(0,\Sigma_2) $$
as $n\rightarrow +\infty$, where $\Xi_2, \Sigma_2$ are defined by (\ref{eq.35}) and (\ref{eq.29}), respectively.
\end{thm}
The proof of Theorem \ref{thm.7} is given in Subsection \ref{subsection.4.3}.
\subsection{Well-balanced linear fractional stable motion}\label{subsection.3.3}
\begin{defn}\emph{Well-balanced linear fractional stable motion}\\
 Let $M$ be a $S\alpha S$ random measure, $0<\alpha\leq 2$, with Lebesgue control measure and consider
 $$
 X(t)=\int_{-\infty}^{+\infty}(\mid t-x\mid^{H-1/\alpha}-\mid x\mid^{H-1/\alpha})M(dx), -\infty <t<+\infty
 $$
 where $0<H<1, H\neq 1/\alpha$. The process $X$ is called the well-balanced linear fractional stable motion.
 Then $X$ is a $H$-sssi process (Proposition 7.4.2, \cite{Taqqu1994}).
\end{defn}
Let
\begin{equation}\label{eq.13}
b_n=
\begin{cases}
 n^{-1/2}&\mbox{ if } H<L+1-\frac{2}{\alpha} \\ 
 n^{\frac{\alpha H-(L+1)\alpha}{4}}&\mbox{ if } H>L+1-\frac{2}{\alpha} \\ 
 \sqrt{\frac{\ln n}{n}}&\mbox{ if } H=L+1-\frac{2}{\alpha}. 
\end{cases}
\end{equation}
It is clear that $\lim\limits_{n\rightarrow +\infty}b_n=0$ and $b_{n/2}=O(b_n)$. We get the following results for the estimation of $H$ and $\alpha$.
\begin{thm}\label{thm.8}
Let $\{X(t)\}_{t\in{\mathbb{R}}}$ be a well-balanced linear fractional stable motion with 
$0<H<1, H\neq 1/\alpha $ and $0<\alpha<2$. %,-1/2<\beta<0$.  
Then for every $\beta\in (-1/2,0)$, Theorem \ref{thm.5} is true with
$b_n$ defined by (\ref{eq.13}).
\end{thm}
See Subsection \ref{subsection.4.3} for the proof of Theorem \ref{thm.8}.
\subsection{Takenaka's processes}\label{subsection.3.4}
\begin{defn}\emph{Takenaka's process}\\
Let $M$ be a symmetric $\alpha-$ stable random measure $(0<\alpha<2)$ with control measure
\[m(dx,dr)=r^{\nu-2}dxdr, (0<\nu<1).\]
Let $t\in \mathbb{R}$, set 
\[C_t=\{(x,r)\in \mathbb{R}\times \mathbb{R}^+, |x-t|\leq r\}, S_t=C_t \Delta C_0\]
where $\Delta $ denotes the symmetric difference between two sets.\\
Takenaka's process is defined by 
\begin{equation}\label{eq.15}
X(t)=\int\limits_{\mathbb{R}\times\mathbb{R}^+}\mathbbm{1}_{S_t}(x,r)M(dx,dr).
\end{equation}
\end{defn}
Following Theorem 4 in \cite{Takenaka1991}, the process $X$ is $\nu/\alpha-$sssi.
Let
\begin{equation}\label{eq.16}
b_n=n^{\frac{\nu-1}{2}}.
\end{equation}
We can now ascertain the following.
%It is clear that $\lim\limits_{n\rightarrow +\infty}b_n=0$ and $b_{n/2}=O(b_n)$.
\begin{thm}\label{thm.9}
 Let $\{X_t, t\in\mathbb{R}\}$ be a Takenaka's process defined by (\ref{eq.15}).
  Then for every $\beta, \beta\in(-1/2,0)$, Theorem \ref{thm.5} is true with
$b_n$ defined by (\ref{eq.16}).
\end{thm}
The proof of Theorem \ref{thm.9} is given in Subsection \ref{subsection.4.3}.
\section{Proofs}\label{section.4}
First, we give results on expectation of negative power variations of $H$-sssi, $S\alpha S$ random processes in Subsection \ref{subsection.4.1}. Then we apply these results in Subsection \ref{subsection.4.2} to the estimation of $H$ and $\alpha$, in order to prove Theorem \ref{thm.5}. Finally, we prove that Theorem \ref{thm.5} is true for four classical examples presented in Section \ref{section.3}.
\subsection{Negative power expectation and auxiliary results}\label{subsection.4.1}
Now we present some results on expectation of negative power variations of $H$-sssi, $S\alpha S$ random processes proved by using theory of distribution. %These results will be used for examples.\\
These results are the tools to prove assumptions (\ref{eq.9.2}) for four examples in Section \ref{section.3} and to prove the main result on the estimation for $\alpha$. 
\subsubsection{Auxiliary results}\label{subsubsection.4.1.1}
We start with the following lemma which confirms the existence of the expectation of $\beta$-negative power variation of a symmetric stable random variable when $ \beta\in \mathbb{C},Re(\beta)\in (-1,0)$.
 \begin{lem}\label{lem.0}
  Let $X$ be a $S\alpha S$ random variable, $\beta\in\mathbb{C}, Re(\beta)\in (-1,0)$, then $|\mathbb{E}|X|^{\beta}|<+\infty$.
 \end{lem}
 The proof of Lemma \ref{lem.0} is given in Subsection \ref{subsubsection.4.1.2}\\
 The next two important results will be used to prove the condition (\ref{eq.9.2}) for our examples in next section. Theorem \ref{thm.1} gives a way to determine the expectation of $\beta$-negative power variation of a symmetric stable random variable whereas Theorem \ref{thm.2} helps to establish the inequality of (\ref{eq.9.2}) for illustrated examples. \\
Let $(S,\mu)$ be a measure space, $h,g\in L^\alpha(S,\mu)$ and $M$ be a symmetric $\alpha$-stable random measure on $S$ with control measure $\mu$, $\alpha\in(0,2)$. Set
 \begin{align}
  U&=\int\limits_Sh(s)M(ds),
  V=\int\limits_Sg(s)M(ds).\label{eq.01}
 \end{align}
 Let 
\begin{equation}\label{eq.1}  
  C_{\beta}=\frac{2^{\beta+1/2}\Gamma(\frac{\beta+1}{2})}{\Gamma(-\frac{\beta}{2})}
  \end{equation}
  where $\beta\in \mathbb {C}$ such that $ Re(\beta)\in(-1,0)$.
   \begin{thm}\label{thm.1}
 For $\beta\in \mathbb{C}, Re(\beta)\in(-1,0)$, we have
\begin{equation}\label{eq.2}
 \mathbb{E}| U|^\beta=
 \frac{1}{\sqrt{2\pi}}\int\limits_\mathbb{R}\mathcal{F}T(y)\mathbb{E}e^{iUy}dy= \frac{C_\beta}{\sqrt{2\pi}}\int\limits_\mathbb{R}\frac{\mathbb{E}e^{iUy}}{|y|^{\beta+1}}dy
\end{equation}
in the sense of distributions, where $U, V$ are defined by (\ref{eq.01}), $T= |x|^\beta$ and $\mathcal{F}T$ is Fourier transform of $T$.% which is defined as in \ref{eq.1.3}.
\end{thm}
See Subsection \ref{subsubsection.4.1.3} for the proof of Theorem \ref{thm.1}.
\begin{thm}\label{thm.2}
Assume that
\begin{align*}
 ||U||_\alpha^\alpha&=\int\limits_S| h(s)|^\alpha\mu(ds)=1,
 ||V||_\alpha^\alpha=\int\limits_S| g(s)|^\alpha\mu(ds)=1\\
 [U,V]_2&=\int\limits_S | h(s)g(s)|^{\alpha/2}\leq \eta<1,
\end{align*}
where $U, V$ are defined as in (\ref{eq.01}). Then for $-1/2<Re(\beta)<0$, we have
 \begin{equation}\label{eq.3}
 \mathbb{E}| U|^\beta| V|^{\overline{\beta}}= \frac{C_\beta C_{\overline{\beta}}}{2\pi}
 \int\limits_{\mathbb{R}^2}\frac{\mathbb{E} e^{ixU+iyV}}{|x|^{1+\beta}| y|^{1+\overline{\beta}}}dxdy.
 \end{equation}
 Moreover, there exists a constant $C(\eta)$ such that
 \begin{equation}\label{eq.3.1}
  |cov(| U|^\beta,| V|^\beta)|\leq C(\eta)\int\limits_S| h(s)g(s)|^{\alpha/2}ds.
 \end{equation}
\end{thm}
The proof of Theorem \ref{thm.2} is given in Subsection \ref{subsubsection.4.1.4}.\\
The following two lemmas follow from Theorem \ref{thm.1} in which  Lemma \ref{lem.01} provides an important formula to construct the estimator for $\alpha$.% and Lemma \ref{lem.9} proves  
\begin{lem}\label{lem.01}
%(Sua lai  p trong dinh ly nay boi $\beta$)
Let X be a standard $S\alpha S$ variable with $0<\alpha\leq 2$ and $\beta \in\mathbb{C}, -1<Re(\beta)<0$, then
\begin{equation}\label{eq.9}
\mathbb{E}| X |^\beta=\frac{2^{\beta}\Gamma(\frac{\beta+1}{2})\Gamma(1-\frac{\beta}{\alpha})}
{\sqrt{\pi}\Gamma(1-\frac{\beta}{2})}.
\end{equation}
\end{lem}
See Subsection \ref{subsubsection.4.1.5} for the proof of Lemma \ref{lem.01}.
\begin{lem}\label{lem.9}
 Let $X$ be a $S\alpha S$ process where $0<\alpha\leq 2$, $\beta\in\mathbb{C},-\frac{1}{2}<Re(\beta)<0$, then
 $$
 \mathbb{E}| \triangle_{0,1}X|^\beta \neq 0.
 $$
\end{lem}
See Subsection \ref{subsubsection.4.1.6} for the proof of Lemma \ref{lem.9}.
Now we will give the proofs for the latter results.
\subsubsection{Proof of Lemma \ref{lem.0}}\label{subsubsection.4.1.2}
%\begin{proof}[{\bf{Proof of Lemma \ref{lem.0}}}]
 Since $X$ is a $S\alpha S$ -stable random variable, $X$ has a density function $f(x)$ that is even and continuous on $\mathbb{R}$.
%$X$ is symmetric then $f(x)$ is an even function.% $f(x)=f(-x)$.%. or in other words $f(x)$ is an even function.
We first consider the case $\beta\in \mathbb {R} $ and $-1<\beta<0$.\\
For $-1<\beta<0$, we can write:
\begin{align*}
 \mathbb{E}| X|^\beta &=\int\limits_\mathbb{R} |x|^\beta f(x)dx =\int\limits_{| x| \leq 1} |x| ^\beta f(x)dx+\int\limits_{|x|\geq 1} | x| ^\beta f(x)dx:=A+B.
\end{align*}
We have
$$ A = \int\limits_{| x| \leq 1} |x| ^\beta f(x)dx \leq 
\sup_{|x|\leq 1}| f(x)| \int_{| x| \leq 1} |x| ^\beta dx <+\infty,
 B =
%\int_1^\infty x^\beta f(x) dx+\int_{-\infty}^{-1} (-x)^\beta f(x) dx=
2\int\limits_1^\infty x^\beta f(x) dx\leq 2.
$$
It follows that $\mathbb{E}| X|^\beta<+\infty$.
%if $-1<\beta<0$. 
For $\beta\in \mathbb{C}, -1<Re(\beta)<0$, we have
%we can write $\beta=a+ib, -1<a<0$. Then
$$
\left|\mathbb{E}|X|^\beta\right| =\left| \int\limits_\mathbb{R}| x|^{a+ib}f(x) dx\right|
\leq \int\limits_\mathbb{R}|x|^{Re(\beta)} f(x) dx<+\infty .
$$
Then we obtain the conclusion.
%\end{proof}
\subsubsection{Proof of Theorem \ref{thm.1}}\label{subsubsection.4.1.3}
To prove Theorem \ref{thm.1}, we start with the following lemma.
 \begin{lem}\label{lem.1}
  For all $x\in\mathbb{R}, \beta\in\mathbb{C}, -1<Re(\beta)<0$, let $T(x)=| x|^\beta$.
Then $T$ has Fourier transform defined by 
\begin{equation}\label{eq.4}
\mathcal{F}{T}(y)=\frac{C_{\beta}}{| y|^{\beta+1}}
\end{equation}
in the sense of distributions, where $C_{\beta}$ is defined as in (\ref{eq.1}).
\end{lem}
\begin{proof}
For $\beta\in\mathbb{C}, -1<Re(\beta)<0$, following example 5, §7, chapter VII of \cite{Schwartz1978}, 
 then $T$ is a distribution and it has Fourier transform
 $\mathcal{F}{T}(y)=C| y|^{-(\beta+1)},$
 where $C$ is a constant.
 We will find $C$ using function $k(x)=e^{-x^2/2}$. 
 Since $T\in L^1_{loc}( \mathbb{R})$ and $k\in S(\mathbb{R}) $, in the sense of distributions, we have
 $\langle \mathcal{F}{T}, k\rangle=\langle T, \mathcal{F}{k} \rangle.$
 On the other hand,
 \begin{align*}
  \mathcal{F}{k}(y)&=\frac{1}{\sqrt{2\pi}}\int\limits_{\mathbb {R}}e^{-ixy}e^{-x^2/2}dx=e^{-y^2/2}
  %&=\frac{1}{\sqrt{2\pi}}\int_{\mathbb {R}}e^{-(\frac{x+iy}{\sqrt{2}})^2}e^{-y^2/2}dx\\
  %&=e^{-y^2/2}\frac{1}{\sqrt{2\pi}}\int_{\mathbb {R}}e^{-t^2/2}dt=e^{-y^2/2}\\
 \end{align*}
 then
 \begin{align*}
  \int\limits_{\mathbb{R}}| x| ^\beta e^{-x^2/2}dx &=\int\limits_{\mathbb{R}} C |y|^{-(\beta+1)}e^{-y^2/2}dy.
 \end{align*}
 By taking the change of variable, we obtain that
 $$
  \int\limits_{\mathbb{R}}| x| ^\beta e^{-x^2/2}dx=%2 \int\limits_0^{\infty} x^\beta e^{-x^2/2}dx=
  2^{\frac{\beta+1}{2}}\Gamma(\frac{\beta +1}{2}),
  \int\limits_{\mathbb{R}}| y| ^{-(\beta+1)} e^{-y^2/2}dy=2^{-\beta/2}\Gamma(-\frac{\beta}{2}).
$$
It follows that
$$
C=\frac{2^{\beta+1/2}\Gamma(\frac{\beta+1}{2})}{\Gamma(-\frac{\beta}{2})}=C_{\beta}.
$$
\end{proof}
From Lemma \ref{lem.1}, we have
 $\mathcal{F}T(y)=\frac{C_{\beta}}{| y|^{1+\beta}}$,
 where $f$ is the density function of $U$ and
 $C_u=2^{u+1/2}\frac{\Gamma(\frac{u +1}{2})}{\Gamma(-\frac{u}{2})}
 $.\\
  Let $\varphi$ be a non-negative, even function, $\varphi\in C_0^\infty(\mathbb{R}),
 supp\varphi\subset[-1,1],\int\limits_{\mathbb{R}}\varphi(y)dy=1$. Set 
$\varphi_\epsilon (x)=\frac{\varphi(x/\epsilon)}{\epsilon}$,
we will prove that
$g_{\epsilon}=\mathcal{F}^{-1}f*\varphi_\epsilon \in S(\mathbb{R}). $\\
 Indeed, let $\chi(x)$ be a function in $C_0^\infty(\mathbb{R})$ such that $\chi (x)=1$ 
 for $|x|\leq 1 $ and $\chi (x)=0$ 
 for $| x|\geq 2 $.\\
 We can write the characteristic function corresponding with the density function $f$ as
 $$
 e^{-\sigma^{\alpha}| x|^\alpha}=\sqrt{2\pi}\mathcal{F}^{-1}f(x):= g(x)=\chi (x)g(x)+(1-\chi (x))g(x):=g_1(x)+g_2(x)
 $$
 and
 $$
 g*\varphi_\epsilon (x)=g_1*\varphi_\epsilon (x)+g_2*\varphi_\epsilon (x).
 $$
 It is clearly that $g_1\in L^1 (\mathbb{R})$, $g_1$ has compact support, $\varphi_{\epsilon}\in C_0^{\infty}(R)$,
  so $g_1*\varphi_\epsilon\in S(\mathbb{R})$.\\
 We also have $g_2* \varphi\in S(\mathbb{R})$ since $g_2 $ and $\varphi_\epsilon (x)$ are in $S(\mathbb {R})$.\\
 Then we get $g_\epsilon\in S(\mathbb{R}) $.\\
We have
$$
\mathcal{F}g_\epsilon(x)=\sqrt{2\pi}f(x)\mathcal{F}\varphi_\epsilon(x).
$$
Since
$$
\langle T, \mathcal{F}g_\epsilon\rangle=\langle \mathcal{F}T,g_\epsilon\rangle,
$$
we obtain
\begin{align}
\int\limits_\mathbb{R}\sqrt{2\pi}| x|^\beta f(x)\mathcal{F}\varphi_\epsilon(x)dx&=
\int\limits_\mathbb{R}\mathcal{F}T(y)\mathcal{F}^{-1}f*\varphi_\epsilon (y) dy \label{eq.5}\\
&=\int\limits_\mathbb{R}\mathcal{F}^{-1}f(y)\mathcal{F}T*\varphi_\epsilon (y) dy,\label{eq.6}
\end{align}
Here we used Fubini's theorem since 
$\mathcal{F}T,\mathcal{F}^{-1}f\in L_{loc}^1(\mathbb{R}), \varphi_\epsilon\in C_0^\infty (\mathbb R)$ and
$\varphi_\epsilon$ is an even function.\\
Now we will find the limits of two sides of the equation (\ref{eq.6}) when $\epsilon\rightarrow 0$.
We first consider the left hand side of (\ref{eq.6}). One has
\begin{align*}
\lim\limits_{\epsilon\rightarrow 0}\mathcal{F}\varphi_\epsilon (x)&=
\lim\limits_{\epsilon\rightarrow 0}\int\limits_\mathbb{R} \frac{1}{\sqrt{2\pi}}e^{-itx}\varphi_\epsilon (t)dt
=\lim\limits_{\epsilon\rightarrow 0}\int\limits_\mathbb{R} \frac{1}{\sqrt{2\pi}}e^{-itx}
\frac{\varphi (t/\epsilon)}{\epsilon} dt
=\lim\limits_{\epsilon\rightarrow 0}\int\limits_\mathbb{R} \frac{1}{\sqrt{2\pi}}e^{-i\epsilon u x}\varphi (u) du.
\end{align*}
For $x,u\in\mathbb {R}, e^{-i\epsilon ux}\varphi (u)\rightarrow \varphi (u) $ when $\epsilon\rightarrow 0$,
and $| e^{-i\epsilon u x}\varphi(u)|=\varphi(u), \int\limits_\mathbb{R}\varphi (u) du=1$,
following Lebesgue dominated convergence theorem, one gets
$$
\lim\limits_{\epsilon\rightarrow 0}\mathcal{F}\varphi_\epsilon (x)=
\int\limits_\mathbb{R} \frac{1}{\sqrt{2\pi}}\varphi (u)du=\frac{1}{\sqrt{2\pi}}.
$$
Therefore, for $x\neq 0$,
$
\sqrt{2\pi}| x|^\beta f(x)\mathcal{F}\varphi_\epsilon(x)\rightarrow 
\sqrt{2\pi}| x|^\beta f(x)
$
pointwise when $\epsilon\rightarrow 0$. We have
\begin{align*}
 \left| | x|^\beta f(x)\mathcal{F}\varphi_\epsilon(x)\right|&=
 \frac{1}{\sqrt{2\pi}}| x|^{Re(\beta)}f(x)| \int\limits_{\mathbb{R}}e^{-itx}\frac{\varphi(t/\epsilon)}{\epsilon} dt|
 =\frac{1}{\sqrt{2\pi}}|x|^{e(\beta)}f(x)
 \left| \int\limits_{\mathbb{R}}e^{-i\epsilon u x}\varphi(u)du \right|\\
 &\leq \frac{1}{\sqrt{2\pi}}| x|^{Re(\beta)}f(x)
 \left| \int\limits_{\mathbb{R}}\varphi(u)du \right|
 =\frac{1}{\sqrt{2\pi}}|x|^{Re(\beta)}f(x).
\end{align*}
Moreover, applying Lemma \ref{lem.0}, it follows that $\int\limits_{\mathbb{R}}| x|^{Re(\beta)}f(x)dx<\infty$.
Thus applying Lebesgue dominated convergence theorem, the left hand side of (\ref{eq.6}) converges to 
$\int\limits_\mathbb{R}|x| ^{Re(\beta)} f(x)dx$.\\
Turning back to the right hand side of (\ref{eq.6}), since  $\mathcal{F}T$ is continuous at $y\neq 0$ and 
  $\mathcal{F}T \in L^1_{loc}(\mathbb{R})$, we get $$\lim\limits_{\epsilon \rightarrow 0} \mathcal{F}T*\varphi_\epsilon(y)=
\mathcal{F}T(y)$$ for $y\in\mathbb{R}^*$. It follows that
  $$
\lim\limits_{\epsilon \rightarrow 0} \mathcal{F}^{-1}f(y)\mathcal{F}T*\varphi_\epsilon(y)=
\mathcal{F}^{-1}f(y)\mathcal{F}T(y)
$$
pointwise almost everywhere. We have the following inequality on $\mathcal{F}T$ and $\varphi_\epsilon$.
\begin{lem}\label{lem.3}
%Let $T$ be defined as in theorem \ref{thm.1}, $\varphi_\epsilon$ be defined as in lemma \ref{lem.0}.
There exists a constant $C>0$ such that for all $\epsilon>0, x\neq 0$, we have
$$| \mathcal{F}T| *\varphi_\epsilon (x)\leq C|\mathcal{F} T| (x).$$
\end{lem}
\begin{proof}
  Since $\mathcal{F}T$ and $\varphi_\epsilon$ are even functions, we just need to prove this lemma for $x>0$. From the fact that $\varphi$ has compact support, then there exists a constant $C$ such that for all $x>0$,
 $
 \varphi(x)\leq C \mathbbm{1}_{[ -1,1]}(x).
 $\\
 We consider first the case $x>2\epsilon$. One has 
 \begin{align*}
 | \mathcal{F}T| *\varphi_\epsilon(x)&=\int\limits_{\mathbb{R}}| \mathcal{F}T| (y)\varphi_\epsilon (x-y)dy
  \leq \frac{C}{\epsilon}
 \int\limits_{x-\epsilon}^{x+\epsilon}| \mathcal{F} T| (y) dy\\
 %&%\leq\frac{C}{\epsilon}
%2\epsilon | \mathcal{F} T| (x-\epsilon)\\
&\leq 2C| \mathcal{F}T| (x-\epsilon)\leq 2C | \mathcal{F}T| (x/2)=C_1| \mathcal{F}T| (x).
 \end{align*}
 If $x\leq 2\epsilon$, then
 \begin{align*}
  |\mathcal{F}T| *\varphi_\epsilon(x)&\leq \frac{C}{\epsilon}
  \int\limits_{x-\epsilon}^{x+\epsilon} | \mathcal{F}T| (y)dy
  \leq  \frac{C}{\epsilon}
  \int\limits_{-3\epsilon}^{3\epsilon} | \mathcal{F}T| (y)dy\\  
  &=\frac{2C}{\epsilon}\int\limits_0^{3\epsilon} \frac{| F_\beta|}{| y|^{1+Re(\beta)}}dy
  =\frac{C_1}{\epsilon}\frac{1}{(3\epsilon)^{Re(\beta)}}%=\frac{C_2}{(3\epsilon)^{1+Re(\beta)}}\\
  \leq C_2| \mathcal{F}T | (x).
 \end{align*}
 \end{proof}
Applying Lemma \ref{lem.3}, then we deduce that
$$
| \mathcal{F}^{-1}f(y)\mathcal{F}T*\varphi_\epsilon(y)| \leq
C| \mathcal{F}^{-1}f(y)|| \mathcal{F}T| (y)
$$
almost everywhere.
But
\begin{align*}
 \int\limits_\mathbb{R}| \mathcal{F}^{-1}f(y)| |\mathcal{F}T| (y) dy&=
 \int\limits_\mathbb{R}\frac{e^{-|y|^\alpha}}{\sqrt{2\pi}}
 \frac{|C_\beta|}{|y|^{1+Re(\beta)}}dy
 =\frac{| C_\beta|}{\sqrt{2\pi}}\int\limits_\mathbb{R}\frac{e^{-|y|^\alpha}}{| y|^{1+Re(\beta)}}dy<\infty
\end{align*}
since $Re(\beta)\in (-1,0)$.
Applying Lebesgue dominated convergence theorem again, the right hand side of 
(\ref{eq.6}) converges to $\int\limits_\mathbb{R}\mathcal{F}^{-1}f(y)\mathcal{F}T(y)dy$.
So we get (\ref{eq.2}).
 % \end{proof}
  \subsubsection{Proof of Theorem \ref{thm.2}}\label{subsubsection.4.1.4}
  Let $\chi$ be in $C_0^\infty(\mathbb{R}), \chi\geq 0,\chi(x)=1$ if
$x\in [-1,1]$, $ supp\chi\in [-2, 2]$.
For $\epsilon >0$, we define 
$$
\phi_\epsilon(x)=(1-\chi(x/\epsilon))\chi(\epsilon x).
$$
Let $\mu$ be the distribution of random vector $(U,V)$, then $\mu$ is a probability measure on $\mathbb{R}^2$.\\
Let $T_1(x)=| x|^\beta, T_2(x)=|y| ^{\overline{\beta}}$.
Following Lemma \ref{lem.1}, % and example 5, §7, chapter VII of \cite{Schwartz},
  $T_1, T_2$ are distributions and have Fourier transforms
$$
\mathcal{F}T_1(y)=\frac{C_{\beta}}{|y|^{1+\beta}},
\mathcal{F}T_2(x)=\frac{C_{\overline{\beta}}}{|x|^{1+\overline{\beta}}} ,
 $$
 respectively, in the sense of distributions, where 
$C_u=2^{u+1/2}\frac{\Gamma(\frac{u +1}{2})}{\Gamma(\frac{-u}{2})}.
 %F_{\overline{\beta}}=2^{\overline{\beta}+1/2}\frac{\Gamma(\frac{\overline{\beta} +1}{2})}{\Gamma(\frac{-\overline{\beta}}{2})}.
$\\
Set
$
F_{1\epsilon}(x)=T_1(x)\phi_\epsilon(x), F_{2\epsilon}(y)=T_2(y)\phi_\epsilon(y).
$
It is clearly that $F_{1\epsilon}(x)\in S(\mathbb{R}), F_{2\epsilon}(y)\in S(\mathbb{R})$.\\
Then $ F_{1\epsilon}\otimes F_{2\epsilon}(x,y)\in S(\mathbb{R}^2)$.
It follows that
\begin{equation}\label{eq.7}
 \int\limits_{\mathbb{R}^2}F_{1\epsilon}\otimes F_{2\epsilon}(x,y)d\mu(x,y)=
 \int\limits_{\mathbb{R}^2}\mathcal{F}^{-1}(d\mu)(x,y)
 \mathcal{F}(F_{1\epsilon}\otimes F_{2\epsilon})(x,y)dxdy.
\end{equation}
Now we consider the right-hand side of (\ref{eq.7}).
We have
\begin{align*}
\mathcal{F}(F_{1\epsilon}\otimes F_{2\epsilon})(x,y)&
=\mathcal{F}F_{1\epsilon}\otimes \mathcal{F}F_{2\epsilon}(x,y).
\end{align*}
We can write
\[
F_{1\epsilon}(x)=T_1(x)\chi(\epsilon x)-T_1(x)\chi(\epsilon x)\chi(\frac{x}{\epsilon}).
\]
Set $\psi(x)=\mathcal{F}\chi(x)$. One has
\begin{align*}
\mathcal{F}F_{1\epsilon}(x)&=\frac{1}{\sqrt{2\pi}}\mathcal{F}T_1*\psi_\epsilon-
\frac{1}{2\pi}\mathcal{F}T_1*\psi_\epsilon *\psi_{1/\epsilon}
=\mathcal{F}T_1*(\frac{1}{\sqrt{2\pi}}\psi_\epsilon)-
\frac{1}{\sqrt{2\pi}}\mathcal{F}T_1*(\frac{1}{\sqrt{2\pi}}\psi_\epsilon) *\psi_{1/\epsilon}.
\end{align*}
We will use the following lemma.
\begin{lem}\label{lem.6}
  Let $\psi$ be a function in the Schwartz class, $T(t)=\mid t\mid^{\beta}$ where $ Re(\beta)\in (-1,0)$. Then for all $x\neq 0$, there exists a constant $C>0$ such that 
 $$
 \mid T*\psi(x) \mid \leq% C \sup_{a>0}\frac{1}
 %{2a}\int_{x-a}^{x+a}\mid T(t) \mid dt\leq
  C\mid T(x)\mid.
 $$
 \end{lem}
\begin{proof}
We denote $C$ a running constant which may change from an occurrence to another occurrence. For $\epsilon>0$, set
 $$\psi_\epsilon(x)=\frac{1}{\epsilon}\psi(\frac{x}{\epsilon}).$$
%Then for the first inequality, it suffices to 
We first prove that there exists a constant $C>0$ such that for all $\epsilon>0$,
  $$
 \mid T*\psi_\epsilon(x) \mid \leq C \sup_{a>0}\frac{1}{2a}\int_{x-a}^{x+a}\mid T(t) \mid dt.
 $$
Let
$k(y)=\mid T(x-y) \mid,I_\epsilon= T*\psi_\epsilon(x)$.\\
By taking the change of variable $u=\frac{y}{\epsilon}$, we obtain that 
$I_\epsilon =\int\limits_{\mathbb {R}} k(\epsilon u) \psi(u) du.$
Set $F(x)=\int\limits_0^xk(\epsilon u) du$, one has 
$$F(x)=\frac{1}{\epsilon}\int_0^{\epsilon x} k(t) dt,
F'(x)=k(\epsilon x).$$
Combining with the fact that
$\lim\limits_{x\rightarrow \infty}F(x)\psi(x)=0$
and $F(0)=0$, we deduce that 
\begin{align*}
 \int\limits_0^{+\infty} k(\epsilon u) \psi(u) du&=\int_0^{+\infty}F'(u)\psi(u) du
 =-\int\limits_0^{+\infty}F(u)\psi'(u) du\\
 &=-\int\limits_0^{+\infty}\left\{\frac{1}{\epsilon u}\int\limits_0^{\epsilon u} k(t)dt\right\}u\psi'(u) du.
\end{align*}
Since $k(u)\geq 0$, it follows that
$$
\frac{1}{\epsilon u}\int\limits_0^{\epsilon u} k(t)dt\leq 
\sup_{a>0}\frac{1}{a}\int\limits_{-a}^{+a} k(t) dt.
$$
We also have $\psi $ is a function in the Schwartz class, then  
$$ |\int\limits_0^{+\infty} k(\epsilon u) \psi(u) du| \leq |\int\limits_0^{+\infty} u\psi'(u) du|
\sup_{a>0}\frac{1}{a}\int_{-a}^{+a} k(t) dt=C\sup_{a>0}\frac{1}{2a}\int\limits_{-a}^{+a} k(t) dt.
$$
We can get a similar bound for the integral  $ |\int\limits_{-\infty}^0 k(\epsilon u) \psi(u) du| $.
Therefore we obtain
$$|I_\epsilon |\leq C \sup_{a>0}\frac{1}{2a}\int\limits_{x-a}^{x+a}\mid T(t) \mid dt.$$
Taking $\epsilon = 1$, it follows that 
$$
\mid T*\psi(x) \mid \leq C \sup_{a>0}\frac{1}
 {2a}\int\limits_{x-a}^{x+a}\mid T(t) \mid dt.
$$
Now we will prove that there exists a constant $C>0$ such that for all $a>0$, then 
$$
\frac{1}
 {2a}\int\limits_{x-a}^{x+a}\mid T(t) \mid dt\leq C|T(x)|.
$$
 We first consider the case $x>0$.\\
 If $x>2a$, then $0<\frac{x}{2}<x-a<x+a$ and $T(t)$ decreases over $[x-a,x+a]$. We get
\begin{align*}
\frac{1}{2a}\int\limits_{x-a}^{x+a}\mid T(t) \mid dt&\leq \frac{1}{2a}\left(x+a-(x-a)\right)(x-a)^{Re(\beta)}\\
%&=
%\frac{(x-a)^{Re(\beta)}}{2}\\
&\leq \frac{(x/2)^{Re(\beta)}}{2}=C\mid T(x)\mid.
\end{align*}
If $0<x\leq 2a<3a  $, then
\begin{align*}
 \frac{1}{2a}\int\limits_{x-a}^{x+a}\mid T(t) \mid &\leq\frac{1}{2a}\int\limits_{-3a}^{3a}\mid T(t) \mid dt\\
 &=\frac{1}{a}\int_{0}^{3a} t^{Re(\beta)}dt
 =\frac{(3a)^{1+Re(\beta)}}{a(1+Re(\beta))}\\
& \leq C(3a)^{Re(\beta)}\leq Cx^{Re(\beta)}=C|T(x)|.
\end{align*}
For the case $x<0$, if $ x\leq -2a $, then $x-a<x+a<x/2<0$, we obtain
\begin{align*}
\frac{1}{2a}\int\limits_{x-a}^{x+a}\mid T(t) \mid&\leq \frac{(x+a-(x-a))|x+a|^{Re(\beta)}}{2a}\\
&\leq \mid x/2\mid^{Re(\beta)}=C\mid T(x)\mid.
\end{align*}
If $-2a< x<0$, then $-3a<x-a<x+a<3a$, one gets
\begin{align*}
\frac{1}{2a}\int\limits_{x-a}^{x+a}\mid T(t) \mid &\leq \frac{1}{2a}\int\limits_{-3a}^{3a}\mid T(t)\mid dt\\
&=
\frac{1}{a}\int\limits_0^{3a}t^{Re(\beta)}\leq C\mid x\mid^{Re(\beta)}=C\mid T(x)\mid.
\end{align*}
One can therefore obtain the conclusion.
\end{proof}
Since 
$\psi_\epsilon*\psi_{1/\epsilon}\in S(\mathbb{R})$, following Lemma \ref{lem.6}, we have
$$
|\mathcal{F}T_1*(\psi_\epsilon *\psi_{1/\epsilon})(x)| \leq C | \mathcal{F}T_1(x)|. 
$$
Then, there exists a constant $C>0$ such that
$$| \mathcal{F}F_{1\epsilon}(x)| \leq C | \mathcal{F} T_1(x)|.
$$
In a similar way, we also get
$
| \mathcal{F}F_{2\epsilon}(y)| \leq C |\mathcal{F} T_2(y)|.
$
It follows that 
$$|\mathcal{F}^{-1}(d\mu)(x,y)
 \mathcal{F}(F_{1\epsilon}\otimes F_{2\epsilon})(x,y)| 
 \leq C | \mathcal{F}^{-1}(d\mu)(x,y)| | \mathcal{F} T_1(x) \mathcal{F} T_2(y)|.
$$
Let us recall that 
$\int\limits_{\mathbb{R}}\frac{\psi(t)}{\sqrt{2\pi}}dt=\chi(0)=1.$
We will use the two following lemmas to
get
 \begin{align}\label{eq.a}
\lim_{\epsilon\rightarrow 0}\mathcal{F}F_{1\epsilon}(x)=
\mathcal{F}T_1(x), \lim\limits_{\epsilon\rightarrow 0}\mathcal{F}F_{2\epsilon}(y)=
\mathcal{F}T_2(y).
\end{align}
\begin{lem}\label{lem.7}
 Let $T(x)=\mid x\mid^{\beta}$ where $Re(\beta)\in (-1,0)$, $\psi$ be a function in Schwartz class. Then almost everywhere, 
 $$
 \lim\limits_{\epsilon\rightarrow 0}T*\psi_{1/\epsilon}(x)=0.
 $$
 where $\psi_{1/\epsilon}(x)=\epsilon \psi(\epsilon x)$.
\end{lem}
\begin{proof}
Let $x\in\mathbb{R}, x\neq 0$, we have
$$
T*\psi_{1/\epsilon}(x)=\int\limits_{\mathbb{R}}T(y)\epsilon\psi(\epsilon(x-y))dy=
\int\limits_{\mathbb{R}}\mid y\mid^{\beta}\epsilon\psi(\epsilon(x-y))dy.
$$
By taking the change of variable $t=\epsilon y$, one gets
$$
\mid T*\psi_{1/\epsilon}(x)\mid\leq 
\int\limits_{\mathbb{R}}\mid t/\epsilon\mid^{Re(\beta)}\mid \psi(\epsilon x-t)\mid dt
=\epsilon^{-Re(\beta)}\int\limits_{\mathbb{R}}\mid t\mid^{Re(\beta)}\mid \psi(\epsilon x-t)\mid dt.
$$
We write
$$
\int\limits_{\mathbb{R}}\mid t\mid^{Re(\beta)}\mid \psi(\epsilon x-t)\mid dt=
\int\limits_{-1}^1\mid t\mid^{Re(\beta)}\mid \psi(\epsilon x-t)\mid dt
+\int\limits_{\mid t\mid \geq 1}\mid t\mid^{Re(\beta)}\mid \psi(\epsilon x-t)\mid dt:=I_1+I_2.
$$
We consider $I_1$ and $I_2$. Since $\psi$ is a function in Schwartz class, one gets
$\mid\mid \psi\mid\mid_{\infty}<\infty, \mid\mid \psi\mid\mid_1<\infty$. Then
$$
I_1\leq 2\mid\mid \psi\mid\mid_{\infty}\int\limits_0^1t^{Re(\beta)}dt=C<+\infty.
$$
$$
I_2=\int\limits_{\mid t\mid \geq 1}\mid t\mid^{Re(\beta)}\mid \psi(\epsilon x-t)\mid dt
\leq \int\limits_{\mid t\mid \geq 1}\mid \psi(\epsilon x-t)\mid dt\leq \mid\mid \psi \mid\mid_1=C<+\infty.
$$
Since then
$
\mid T*\psi_{1/\epsilon}(x)\mid \leq C\epsilon^{-Re(\beta)}\rightarrow 0
$
as $\epsilon\rightarrow 0$. It follows that $T*\psi_{1/\epsilon}(x)\rightarrow 0$ 
almost everywhere as $\epsilon\rightarrow 0$.
\end{proof}
\begin{lem}\label{lem.8}
Let $\psi$ be a function in Schwartz class such that $\int\limits_{\mathbb{R}}\psi(t) dt=1$,
$T(t)=\mid t\mid^{\beta }$ where $Re(\beta)\in(-1,0)$. Then we have $\lim\limits_{\epsilon\rightarrow 0}T*\psi_\epsilon(x)=T(x)$
 almost everywhere, where $\psi_\epsilon(x)=\frac{\psi(x/\epsilon )}{\epsilon}$.
\end{lem}
\begin{proof}
Let $x\in\mathbb{R}, x\neq 0$, we consider
$
I=T*\psi_\epsilon(x)-T(x).
$
Let us recall that 
$$\int\limits_{\mathbb{R}}\psi_\epsilon(t)dt=
\int\limits_{\mathbb{R}}\frac{\psi(t/\epsilon)}{\epsilon}dt=
\int\limits_{\mathbb{R}}\psi(u)du=1.
$$  
Therefore 
$
I=\int\limits_{\mathbb{R}}\left\{T(x-y)-T(x) \right \}\frac{1}{\epsilon}\psi(y/\epsilon)dy.
$\\
Let $\theta>0$ be a constant. There exists $0<\delta<\mid x\mid$ such that for
$\mid y\mid \leq \delta$, we have $\mid T(x-y)-T(x)\mid \leq \frac{\theta}{2\mid \mid \psi\mid\mid_1}$.
Then
$$
\mid I\mid\leq \int\limits_{\mid y\mid\leq \delta}\mid T(x-y)-T(x)\mid\frac{\mid\psi(y/\epsilon)\mid}{\epsilon}dy
+\int\limits_{\mid y\mid\geq \delta}\mid T(x-y)-T(x)\mid\frac{\mid\psi(y/\epsilon)\mid}{\epsilon}dy
:=I_1+I_2.
$$
$$
I_1\leq \frac{\theta}{2\epsilon\mid\mid\psi\mid\mid_1}
\int\limits_{\mid y\mid\leq \delta}\mid \psi(y/\epsilon)\mid dy
= \frac{\theta}{2\epsilon\mid\mid\psi\mid\mid_1}
\int\limits_{\mid u\mid\leq \epsilon\delta}\epsilon\mid \psi(u)\mid du
\leq \frac{\theta}{2\epsilon\mid\mid\psi\mid\mid_1}
\int\limits_{\mathbb{R}}\epsilon\mid \psi(u)\mid du=\frac{\theta}{2}.
$$
Now we consider $I_2$. Since $\psi$ is a function in Schwartz class, there exists a constant $C>0$ 
such that for $\mid t\mid\geq 1$ then $\mid \psi(t)\mid \leq \frac{C}{t^2}$.\\
We choose $\epsilon>0$ such that $\frac{\delta}{\epsilon}\geq 1$. By taking the  change of variable $t=y/\epsilon$, 
we get
\begin{align*}
 I_2&=\int\limits_{\mid t\mid\geq \frac{\delta}{\epsilon}}\mid T(x-\epsilon t)-T(x)\mid \mid \psi (t)\mid dt
 \leq \int\limits_{\mid t\mid\geq \frac{\delta}{\epsilon}}\mid T(x-\epsilon t)\mid \mid \psi (t)\mid dt
 +\int\limits_{\mid t\mid\geq \frac{\delta}{\epsilon}}\mid T(x)\mid \mid \psi (t)\mid dt\\
 &= \int\limits_{\mid t\mid\geq \frac{\delta}{\epsilon},\mid x-\epsilon t \mid \leq 1 }\mid T(x-\epsilon t)\mid \mid \psi (t)\mid dt+
 \int\limits_{\mid t\mid\geq \frac{\delta}{\epsilon}, \mid x-\epsilon t \mid \geq 1}\mid T(x-\epsilon t)\mid \mid \psi (t)\mid dt+
 \int\limits_{\mid t\mid\geq \frac{\delta}{\epsilon}}\mid T(x)\mid \mid \psi (t)\mid dt\\
 &:=J_1+J_2+J_3.
 \end{align*}
  We have 
 $$
 J_1=\int\limits_{\mid t\mid\geq \frac{\delta}{\epsilon},\mid x-\epsilon t \mid \leq 1 }\mid T(x-\epsilon t)\mid \mid \psi (t)\mid dt
 \leq C\int\limits_{\mid x-\epsilon t\mid \leq 1}\mid T(x-\epsilon t)\mid (\epsilon/\delta)^2 dt.
 $$
 By taking the change of variable $u=\epsilon t-x$, one gets
 $$
 J_1\leq \frac{C\epsilon^2}{\delta^2}\int\limits_{\mid u\mid \leq 1}\frac{\mid T(u)\mid}{\epsilon} du=C_1\epsilon.
 $$
 Here $C_1$ is a constant depending on $\delta$.\\
 Let us consider $J_2$. Since $\mid T(t)\mid =\mid t\mid^{Re(\beta)}$ and $Re(\beta)\in (-1,0)$, 
 if $\mid x-\epsilon t\mid \geq 1 $ we get
 $\mid T(x-\epsilon t)\mid \leq 1$. \\
 Moreover $\delta/\epsilon\geq 1$, it follows that
 $$
 J_2\leq \int\limits_{\mid t\mid \geq \delta/\epsilon}\frac{C}{t^2}dt=C_2\frac{\epsilon}{\delta}
 $$
 where $C_2$ is a constant depending on $\delta$.
Similarly, since $\delta/\epsilon\geq 1$, we get
$$
J_3\leq \mid T(x)\mid \int\limits_{\mid t\mid \geq \delta/\epsilon}\frac{C}{t^2}dt=C_3\epsilon
$$
where $C_3$ is a constant depending on $x, \delta$.
So we get $I_2\leq C\epsilon$ where $C$ is a constant depending on $x, \delta$.
We can choose $\epsilon$ small enough to get $I_2\leq \frac{\theta}{2}$.\\
Then for all $\theta>0$, there exists $\epsilon_0$ such that for all $0<\epsilon<\epsilon_0$, 
we have $\mid I\mid \leq \theta $.
Therefore we get the conclusion.
\end{proof}
From (\ref{eq.a}), one gets
$$
\lim\limits_{\epsilon\rightarrow 0}\mathcal{F}(F_{1\epsilon}\otimes F_{2\epsilon})(x,y)=
\mathcal{F}T_1(x)\mathcal{F}T_2(y)
=\frac{C_\beta C_{\overline{\beta}}}{\mid x\mid^{1+\beta}\mid y\mid ^{1+\overline{\beta}}}.
$$
Moreover
$
\mathcal{F}^{-1}(d\mu)(x,y)=\frac{\mathbb{E}e^{ixU+iyV}}{2\pi}.
$ We use Theorem \ref{thm.1}, Lemma \ref{lem.0} and the following lemma to deduce that
$$
\int\limits_{\mathbb{R}^2}\frac{|\mathcal{F}^{-1}(d\mu)(x,y)|}{\mid x\mid^{1+Re(\beta)}\mid y\mid ^{1+Re(\overline{\beta})}}dxdy=
\int\limits_{\mathbb{R}^2}\frac{\mathbb{E}e^{ixU+iyV}}{2\pi\mid x\mid^{1+Re(\beta)}\mid y\mid ^{1+Re(\overline{\beta})}}dxdy <+\infty.
$$
\begin{lem}\label{lem.5}
  %Let $U,V,\beta$ be defined as in theorem \ref{thm.02}. 
  Set 
  \begin{align*}
  M_{U,V}(x,y)&=\mathbb{E}e^{ixU+iyV}-\mathbb{E}e^{ixU}\mathbb{E}e^{iyV},
  I=\int\limits_{\mathbb{R}^2}\frac{\mid M_{U,V}(x,y)\mid }{\mid xy\mid^{1+Re(\beta)}}dxdy.\\
  [U,V]_2&=\int\limits_S |h(s) g(s)|^{\alpha/2}\leq \eta<1,
  \end{align*}
  where $U, V$ are defined as in Theorem \ref{thm.2}.
  Then
  $
  I\leq C(\eta) [U,V]_2<\infty,
  $
  where the constant $C(\eta)$ depends on $\eta$.
 \end{lem}
  \begin{proof}
We just need to consider the integral only over $(0,+\infty)\times (0, +\infty)$. We divide this domain
into four regions $(0,1) \times (0,1), (0,1)\times (1,+\infty), 
(1,+\infty)\times (0,1), (1,+\infty)\times (1,+\infty)$ and let $I_1, I_2, I_3, I_4$ 
be the integrals over those domains, respectively.\\
Over $(0,1) \times (0,1)$, by using inequality (3.4) in \cite{Pipiras2007} % can bo sung tai lieu tham khao sau
we get
\begin{align*}
 I_1&\leq \int\limits_0^1\int\limits_0^1 \frac{\mid M_{U,V}(x,y)\mid }{\mid xy\mid^{1+Re(\beta)}}dxdy
 \leq 2 \int\limits_0^1\int\limits_0^1  (xy)^{\alpha/2-1-Re(\beta)}dxdy [U,V]_2=C[U,V]_2.
\end{align*}
Over $(1,+\infty)\times (1,+\infty)$, by using inequality (3.6) in \cite{Pipiras2007} % can bo sung tai lieu tham khao sau
and assumptions 
$$
\mid\mid U\mid\mid_\alpha^\alpha=1,  \mid\mid V\mid\mid_\alpha^\alpha=1, [U,V]_2=\int\limits_S\mid f(s)g(s)\mid^{\alpha/2}\leq \eta<1,
$$
we can bound the integral over this domain by
$$
I_2\leq 2 \int\limits_1^{+\infty}\int\limits_1^{+\infty}  (xy)^{\alpha/2-1-Re(\beta)}e^{-2(1-\eta)(xy)^{\alpha/2}}dxdy [U,V]_2.
$$
Here we can bound $e^{-2(1-\eta)(xy)^{\alpha/2}}$ up to a constant depending on $\eta$ by $(xy)^{-p}$ 
for arbitrarily large $p>0$. \\
So 
$
I_2\leq C(\eta) [U,V]_2.
$
Over $(0,1)\times (1,+\infty)$, by using inequality (3.5) in \cite{Pipiras2007} % can bo sung tai lieu tham khao sau
we obtain a bound
\begin{align*}
I_3&\leq 2 \int\limits_0^1\int\limits_1^{+\infty}  (xy)^{\alpha/2-1-Re(\beta)}e^{-(x^{\alpha/2}-y^{\alpha/2})^2}dxdy [U,V]_2\\
&\leq 2 \int\limits_0^1\int\limits_1^{+\infty}  (xy)^{\alpha/2-1-Re(\beta)}e^{-(y^{\alpha/2}-1)^2}dxdy [U,V]_2\\
&\leq C\int\limits_1^{+\infty} (xy)^{\alpha/2-1-Re(\beta)}e^{-(y^{\alpha/2}-1)^2}dxdy [U,V]_2\\
&=C\int\limits_1^{+\infty}y^{\alpha/2-1-Re(\beta)}e^{-(y^{\alpha/2}-1)^2}dy [U,V]_2.
\end{align*}
Since
$
\int\limits_1^{+\infty}y^{\alpha/2-1-Re(\beta)}e^{-(y^{\alpha/2}-1)^2}dy<+\infty,
$
we get $I_3\leq C[U,V]_2$.\\
A similar bound holds for $I_4$. Then we have the conclusion.
\end{proof}
By Lebesgue dominated convergence theorem, as $\epsilon \rightarrow 0$, the right-hand side of (\ref{eq.7}) converges to 
$$
\frac{C_\beta C_{\overline{\beta}}}{2\pi}
 \int\limits_{\mathbb{R}^2}\frac{\mathbb{E} e^{ixU+iyV}}{\mid x\mid^{1+\beta}\mid y\mid^{1+\overline{\beta}}}dxdy.
$$
Now we consider the left-hand side of (\ref{eq.7}). Since 
$\lim\limits_{\epsilon\rightarrow 0}\phi_\epsilon(x)=1$ 
for all $x\in \mathbb{R}$, it follows  
$$
\lim\limits_{\epsilon\rightarrow 0}F_{1\epsilon}(x)F_{2\epsilon}(y)=\mid x\mid^{\beta}\mid y\mid^{\overline{\beta}}
$$
for all $x\neq 0, y\neq 0$.\\
It is clear that $\mid F_{1\epsilon}(x)\mid \leq C \mid x\mid ^{Re(\beta)},
\mid F_{2\epsilon}(y)\mid \leq C \mid y\mid ^{Re(\overline{\beta})}$. Moreover
$$
\int\limits_{\mathbb{R}^2}\mid x\mid ^{Re(\beta)}\mid y\mid ^{Re(\overline{\beta})}d\mu(x,y)= 
\left| \mathbb{E}|U|^\beta|V|^{\overline{\beta}} \right|
\leq \sqrt{\mathbb{E}|U|^{2\beta}\mathbb{E}|V|^{2\overline{\beta}}}<+\infty
$$
since $Re(\beta)\in (-1/2,0)$.\\
We can therefore apply Lebesgue dominated convergence theorem for the left-hand side of (\ref{eq.7}). As $\epsilon\rightarrow 0$, it converges to
$$\int\limits_{\mathbb{R}^2}\mid x\mid ^{\beta}\mid y\mid ^{\overline{\beta}}d\mu(x,y)=
\mathbb{E}\mid U\mid^\beta \mid V\mid^{\overline{\beta}}.
$$
This proves the result (\ref{eq.3}). Now we prove (\ref{eq.3.1}).\\
Following Theorem \ref{thm.1} and (\ref{eq.3}), 
 for $-1/2<Re(\beta)<0$, we get
 \begin{align*}
\mathbb{E}\mid U\mid^\beta &=\frac{C_\beta}{\sqrt{2\pi}} \int\limits_{\mathbb{R}}
\frac{\mathbb{E}e^{ixU}}{\mid x\mid^{1+\beta}}dx,
\mathbb{E}\mid V\mid^{\overline{\beta}} =\frac{C_{\overline{\beta}}}{\sqrt{2\pi}}\int\limits_{\mathbb{R}}
\frac{\mathbb{E}e^{iyV}}{\mid y\mid^{1+\overline{\beta}}}dy\\
\mathbb{E}\mid U\mid^\beta\mid V\mid^{\overline{\beta}}&=\frac{C_{\beta} C_{\overline{\beta}}}{2\pi}\int\limits_{\mathbb{R}^2}
\frac{\mathbb{E}e^{ixU+iyV}}{\mid x\mid^{1+\beta}\mid y\mid^{1+\overline{\beta}}}dxdy.
\end{align*}
Then 
\begin{align*}
 \left|cov(\mid U\mid^\beta,\mid V\mid^\beta)\right|&=
 \left|\mathbb{E}\mid U\mid^\beta\mid V\mid^{\overline{\beta}}-
 \mathbb{E}\mid U\mid^\beta \mathbb{E}\mid V\mid^{\overline{\beta}}\right|
 =\left|C_{\beta}C_{\overline{\beta}} \int\limits_{\mathbb{R}^2}\frac{\mathbb{E}e^{ixU+iyV}-
 \mathbb{E}e^{ixU}\mathbb{E}e^{iyV}}{\mid x\mid^{1+\beta}\mid y\mid^{1+
 \overline{\beta}}}dx dy\right|\\
 &\leq \mid C_\beta C_{\overline{\beta}}\mid \int\limits_{\mathbb{R}^2}\frac{\left|\mathbb{E}e^{ixU+iyV}-
\mathbb{E}e^{ixU}\mathbb{E}e^{iyV}\right|}{\mid xy\mid^{1+Re(\beta)}} dx dy .
 %&= \mid C_\beta C_{\overline{\beta}}\mid \int\limits_{\mathbb{R}^2}\frac{\left| e^{-||xf+yg||_\alpha^\alpha
% -e^{-||xf||_\alpha^\alpha-||yg||_\alpha^\alpha}}\right|}{\mid x\mid^{1+\beta}\mid y\mid^{1+
 %{\beta}}} dx dy .
 \end{align*}
  Applying Lemma \ref{lem.5}, we obtain (\ref{eq.3.1}).
  \subsubsection{Proof of Lemma \ref{lem.01}}\label{subsubsection.4.1.5}
For the case $\alpha=2$, let $Y$ be a standard $S2S$ variable. Then for $-1<Re(\beta)<0$, we have  
  \begin{align*}
\mathbb{E}\mid Y \mid^\beta &=\frac{1}{2\sqrt{\pi}}\int\limits_{-\infty}^{+\infty}\mid x\mid^\beta
e^{-\frac{x^2}{4}}dx
=\frac{1}{\sqrt{\pi}}\int\limits_0^{+\infty}x^\beta e^{-x^2/4}dx
%&=\frac{2^\beta}{\sqrt{\pi}}\int\limits_0^{+\infty}t^{\beta/2-1/2}e^{-t}dt %{\text{ (using the change of variable  } t=x^2/4) }\\
=\frac{2^\beta}{\sqrt{\pi}}\cdot 
\Gamma(\frac{\beta+1}{2}).
 \end{align*}
  Let us now consider the case $\alpha\neq 2$.
Following (\ref{eq.4}) and Theorem \ref{thm.1}, we have 
\begin{align*}
 \mathbb{E}\mid X\mid^\beta&=
 %\frac{1}{\sqrt{2\pi}}\int\limits_\mathbb{R}\mathcal{F}T(y)Ee^{iUy}dy
 \frac{C_{\beta}}{\sqrt{2\pi}}\int\limits_\mathbb{R}\frac{\mathbb{E}e^{iXy}}{|y|^{\beta+1}}dy
 =\frac{C_{\beta}}{\sqrt{2\pi}}\int\limits_\mathbb{R}\frac{e^{-|y|^{\alpha}}}{|y|^{\beta+1}}dy
 =\frac{2C_{\beta}}{\sqrt{2\pi}}\int\limits_0^{+\infty}\frac{e^{-y^{\alpha}}}{y^{\beta+1}}dy.
 \end{align*}
  By making the change of variable $y^{\alpha}=t$, then %we get $y=t^{1/\alpha}, dy=\frac{t^{1/\alpha-1}}{\alpha}dt$. Then
 \begin{align*}
 \mathbb{E}\mid X\mid^\beta&=
 \frac{2C_{\beta}}{\sqrt{2\pi}}\int\limits_0^{+\infty}t^{-\beta/\alpha-1}e^{-t}dt
 =\frac{2C_{\beta}}{\sqrt{2\pi}}\int\limits_0^{+\infty}t^{-\beta/\alpha-1}e^{-t}dt
 =\frac{\sqrt{2}C_{\beta}\Gamma(-\beta/\alpha)}{\alpha\sqrt{\pi}}.
 \end{align*}
 Since $\Gamma(x+1)=x\Gamma(x)$, one gets 
 $$\mathbb{E}\mid X \mid^\beta=\frac{2^{\beta}\Gamma(\frac{\beta+1}{2})\Gamma(1-\frac{\beta}{\alpha})}
{\sqrt{\pi}\Gamma(1-\frac{\beta}{2})}.$$
\subsubsection{Proof of Lemma \ref{lem.9}}\label{subsubsection.4.1.6}
 From the fact that $X$ is a $S\alpha S$ process, one can write %there exists a $\sigma>0$ such that
 %$\{X(t)\}_{t\in\mathbb{R}}$
 $$
 \triangle_{0,1}X=\sum_{k=0}^Ka_kX(k)\stackrel{(d)}{=}\sigma Y
 $$
 where $\sigma>0 $ and  $Y$ is a standard $S \alpha S $ random variable.
 Then
 $
 \mathbb{E}\mid \triangle_{0,1} X \mid^\beta=\sigma^\beta \mathbb{E}\mid Y\mid^\beta$. Following Theorem \ref{thm.1}, since there doesn't exist any $x\in\mathbb{C}$ such that $\Gamma(x)=0$,
 we deduce that $\mathbb{E}\mid Y\mid^\beta\neq 0$.\\
 Thus $
 \mathbb{E}\mid \triangle_{0,1}X\mid^\beta \neq 0.
 $
\subsection{Proof of Theorem \ref{thm.5}}\label{subsection.4.2}
\begin{proof}[\bf{Proof of Theorem \ref{thm.5}}]
In this proof, we shall denote by $C$ a generic constant which may change from occurrence to occurrence.\\
 %We prove the result for $H$ first.
 We will prove that 
 $W_n (\beta)-\mathbb{E}|\triangle_{0,1}X|^\beta=O_\mathbb{P}(b_n)$ where $b_n$ is defined by (\ref{eq.9.2}).\\
 Indeed, from Lemma \ref{lem.0}, it follows that $\mathbb{E}|\triangle_{0,1}X|^\beta<+\infty$.\\
 Because of $H$-self similarity and stationary increment properties of $X$, one has
\begin{align*}
\triangle_{p,n}X&=\sum_{k=0}^Ka_kX(\frac{k+p}{n})\stackrel{(d)}{=} \sum_{k=0}^{K}\frac{a_k}{n^H}X(k+p)%\stackrel{(d)}{=}
=\sum_{k=0}^{K}\frac{a_k}{n^H}(X(k+p)-X(p))\stackrel{(d)}{=} \sum_{k=0}^{K}\frac{a_k}{n^H}X(k)=\frac{\triangle_{0,1}X}{n^H}.
\end{align*}
We get  
$\mathbb{E}|\triangle_{p,n}X|^\beta=%\mathbb{E}|\frac{\triangle_{0,1}X}{n^H}|^\beta=
\frac{\mathbb{E}|\triangle_{0,1}X|^{\beta}}{n^{\beta H}}
$
and
$\mathbb{E}W_n(\beta)=\mathbb{E}|\triangle_{0,1}X|^\beta.$ Now we will prove that 
$W_n(\beta)\xrightarrow{(\mathbb P)} \mathbb{E}|\triangle_{0,1}X|^\beta$ when $n\rightarrow \infty$.\\
We have
 \begin{align*}
\mathbb{E}|W_n(\beta)|^2&=%\mathbb{E}W_n(\beta)\cdot \overline{W_n(\beta)}\\
%\mathbb{E}[\frac{n^{\beta H}}{n-K+1}\cdot \sum_{p=0}^{n-K}|\Delta_{p,n}X|^\beta
%\cdot \frac{n^{\overline{\beta}H}}{n-K+1}
%\cdot \sum_{p'=0}^{n-K}|\triangle_{p',n}X|^{\overline{\beta}}]\\
\frac{n^{2\beta H}}{(n-K+1)^2}\sum_{p,p'=0}^{n-K}\mathbb{E}|
\triangle_{p,n}X|^\beta|\triangle_{p',n}X|^{\beta}.
\end{align*}
%\stackrel{}{}
Moreover
\begin{align*}
 |\triangle_{p,n}X|^\beta |\triangle_{p',n}X|^{\beta}&
 %\mid\sum_{k=0}^Ka_kX(\frac{k+p}{n})\mid^\beta\cdot
 %\mid\sum_{k=0}^Ka_kX(\frac{k+p'}{n})\mid^{\beta}\\
 \stackrel{(d)}{=}|\sum_{k=0}^K\frac{a_k}{n^H}X(k+p)|^\beta
 |\sum_{k=0}^K\frac{a_k}{n^H}X(k+p')|^{\beta}\\
 %&=\frac{1}{n^{2\beta H}}\mid\sum_{k=0}^Ka_kX(k+p)\mid^\beta\cdot
 %\mid\sum_{k=0}^Ka_kX(k+p')\mid^{\beta}\\
 &=\frac{1}{n^{2\beta H}}|\sum_{k=0}^Ka_p[X(k+p)-X(p')]|^\beta
 |\sum_{k=0}^Ka_k[X(k+p')-X(p')]|^{\beta}\\
 &\stackrel{(d)}{=}\frac{1}{n^{2\beta H}}|\sum_{k=0}^Ka_kX(k+p-p')|^\beta
 |\sum_{k=0}^Ka_kX(k)|^{\beta}\\
 &=\frac{1}{n^{2\beta H}}|\triangle_{p-p',1}X|^\beta | \triangle_{0,1}X|^{\beta}.
\end{align*}
It follows that
\begin{align*}
 \mathbb{E}|\triangle_{p,n}X|^\beta|\triangle_{p',n}X|^{\beta}&=\frac{\mathbb{E}|\triangle_{p-p',1}X|^\beta
 |\triangle_{0,1}X|^{\beta}}{n^{2\beta H}}
 =\frac{\mathbb{E}|\triangle_{k,1}X|^\beta
 |\triangle_{0,1}X|^{\beta} }{n^{2\beta H}}
 \end{align*}
 with $k=p-p'$.
  Thus
 $$\mathbb{E}| W_n(\beta)|^2=\frac{1}{n-K+1}\sum_{|k|\leq n-K}(1-\frac{|k|}{n-K+1})
 \mathbb{E}| \triangle_{k,1}X|^\beta | \triangle_{0,1}X|^{\beta}.$$
 One has
 \begin{align*}
   \mathbb{E}| W_n(\beta)-\mathbb{E}|\triangle_{0,1}X|^\beta|^2&= 
   %\mathbb{E}( W_n(\beta)-\mathbb{E}\mid\Delta_{0,1}X\mid^\beta)(\overline{W_n(\beta)}
   %-\mathbb{E}\mid\Delta_{0,1}X\mid^{\overline{\beta}})\\
   %&=\mathbb{E}\mid W_n(\beta)\mid^2-\mathbb{E}W_n(\beta)\cdot \mathbb{E}|\triangle_{0,1}X|^{\overline{\beta}}
   %-\mathbb{E}\overline{W_n(\beta)}\cdot \mathbb{E}|\triangle_{0,1}X|^{\beta}\\
   %&+\mathbb{E}|\triangle_{0,1}X|^\beta\cdot \mathbb{E}|\triangle_{0,1}X|^{\overline{\beta}}\\
  \mathbb{E}| W_n(\beta)|^2-\mathbb{E}|\triangle_{0,1}X|^\beta
  \mathbb{E}|\triangle_{0,1}X|^{\beta}.
\end{align*}
On the other hand,
since $\mathbb{E}|\triangle_{k,1}X|^{\beta}=\mathbb{E}|\triangle_{0,1}X|^{\beta}$ and %.\\
%We also have
$
 \frac{1}{n-K+1}\sum\limits_{\mid k\mid\leq n-K}(1-\frac{\mid k\mid}{n-K+1})=1
% \frac{1}{n-K+1}[1+2(1-\frac{1}{n-K+1})+\cdots +2(1-\frac{n-K}{n-K+1})]=1.
 $, it follows that
\begin{equation}\label{eq.10}
\mathbb{E}| W_n(\beta)-\mathbb{E}|\triangle_{0,1}X|^\beta|^2=
\frac{1}{n-K+1}\sum_{| k|\leq n-K}(1-\frac{| k|}{n-K+1})
cov(|\triangle_{k,1}X|^\beta,|\triangle_{0,1}X|^\beta).
\end{equation}
Using (\ref{eq.10}) and the assumption (\ref{eq.9.2}), one obtains
$$
\limsup_{n} \frac{1}{b_n^2}\mathbb{E}| W_n(\beta)-E|\triangle_{0,1}X|^\beta|^2\leq \Sigma^2.
$$
For all $\epsilon>0$, applying Markov's inequality and using (\ref{eq.10}), we get
\begin{align*}
 \sup_{n}P(| W_n(\beta)- \mathbb{E}|\triangle_{0,1} X|^\beta|> b_n\frac{\Sigma}{\sqrt{\epsilon}})
&\leq \limsup_{n}\frac{\mathbb{E}| W_n(\beta)-E|\triangle_{0,1}X|^\beta|^2}
{{b_n^2}\frac{\Sigma^2}{\epsilon}}\leq \epsilon.
\end{align*}
It follows that
\begin{align}\label{eq.10.1}
 W_n (\beta)- \mathbb{E}|\triangle_{0,1}X|^\beta=O_{\mathbb{P}}(b_n).
 \end{align}
 In a similar way, combining with the fact that $b_{n/2}=O(b_n)$, one also has 
 \begin{align}\label{eq.10.2}
 W_{n/2} (\beta)- \mathbb{E}|\triangle_{0,1}X|^\beta=O_{\mathbb{P}}(b_n).
 \end{align}
 Now we will prove that $\widehat{H}_n-H=O_\mathbb{P}(b_n)$.\\
Let $\phi: \mathbb{R}^+\times \mathbb{R}^+\rightarrow \mathbb{R}$ be defined by
\begin{align}\label{eq.10.3}
\phi(x,y)=\log_2\frac{x}{y}.
\end{align}
Then
 $\widehat{H}_n-H =\phi(W_{n/2}(\beta), W_n(\beta))$.\\
  We need the following lemma.
 \begin{lem}\label{lem.B.3}
Let $f:D\subset\mathbb{R}^2\rightarrow \mathbb{R}$, be differentiable at a constant vector $(a,b)\in D$.
Let $(X_n,Y_n)$ be random vectors whose ranges lie in $D$ such that
$X_n\xrightarrow{\mathbb {P}} a,Y_n\xrightarrow{\mathbb {P}} b$ 
and $X_n-a=O_\mathbb{P}(b_n),Y_n-b=O_\mathbb{P}(b_n)$ where  $\{b_n\}_n $ is a non-negative sequence and
$b_n\rightarrow 0 $ as $n\rightarrow +\infty$.\\
Then $f(X_n,Y_n)-f(a,b)=O_{\mathbb{P}}(b_n)$.
 \end{lem} 
 \begin{proof}
 Since $f$ is differentiable at $(a,b)$, we can write
\begin{align*}
f(a+h_1,b+h_2)%=f(a,b)+\langle(h_1,h_2),\nabla_f(a,b) \rangle+o(|(h_1,h_2||)\\
&=f(a,b)+h_1\frac{\partial f}{\partial x}(a,b)+h_2\frac{\partial f}{\partial y}(a,b)
+o(||(h_1,h_2||).
\end{align*}
as $||h||=||(h_1,h_2)||\rightarrow 0$.\\
By applying Lemma 2.12 in \cite{VanDerVaart2000} for $R(x,y)=f(a+x,b+y)-f(a,b)-
x\frac{\partial f}{\partial x}(a,b)-y\frac{\partial f}{\partial y}(a,b)$
and the sequence random vector $(X_n-a,Y_n-b)$, we get
\begin{align*}
 f(X_n,Y_n)-f(a,b)&=
 (X_n-a)\frac{\partial f}{\partial x}(a,b)+(Y_n-a)\frac{\partial f}{\partial y}(a,b)+
 o_{\mathbb{P}}(||(X_n-a,Y_n-b)||)\\
 &=(X_n-a)O_{\mathbb{P}}(1)+(Y_n-a)O_{\mathbb{P}}(1)+
 o_{\mathbb{P}}(||(X_n-a,Y_n-b)||)\\
% & \text{ (since $\frac{\partial f}{\partial x}(a,b)$ and 
% $\frac{\partial f}{\partial y}(a,b)$ are constants then 
% $\frac{\partial f}{\partial x}(a,b)=O_{\mathbb{P}}%(1)$ and
% $\frac{\partial f}{\partial y}(a,b)=\large{O}_{\mathbb{P}}(1)$)} \\
 &=b_n O_{\mathbb{P}}(1)O_{\mathbb{P}}(1)+b_n O_{\mathbb{P}}(1)O_{\mathbb{P}}(1)
 +b_n O_{\mathbb{P}}(1)o_{\mathbb{P}}(1)\\
 %&\text{ (using property e) and k) of $o_{\mathbb{P}}$ and $\large {O}_{\mathbb{P}}$)}\\
 &=b_n O_{\mathbb{P}}(1)+b_n O_{\mathbb{P}}(1)+b_no_{\mathbb{P}}(1)\\
 %&\text{ (using property f) and g) of $o_{\mathbb{P}}$ and $\large {O}_{\mathbb{P}}$)}\\
 &=b_n O_{\mathbb{P}}(1)
 %\text{ (using property b) and h) of $o_{\mathbb{P}}$ and $\large {O}_{\mathbb{P}}$)}.\\
  \end{align*}
 \end{proof}
  Applying Lemma \ref{lem.B.3} with $f=\phi$ and vector $(\mathbb{E}\mid\triangle_{0,1}X\mid^\beta,\mathbb{E}\mid\triangle_{0,1}X\mid^\beta)$, combining with (\ref{eq.10.1}), (\ref{eq.10.2}) and the fact that $\phi(\mathbb{E}\mid\triangle_{0,1}X\mid^\beta,\mathbb{E}\mid\triangle_{0,1}X\mid^\beta)=0$, it follows that $\widehat{H}_n-H=O_\mathbb{P}(b_n)$. \\
  Since $\lim\limits_{n\rightarrow+\infty}b_n=0$, it induces that $\lim\limits_{n\rightarrow+\infty}\widehat{H}_n\stackrel{(\mathbb{P})}{=}H$.\\
  To prove that $\widehat{\alpha}_n-\alpha=O_\mathbb{P}(b_n)$, we first prove that
  \begin{align}\label{eq.19.0.1}
h_{-\beta_1,-\beta_2}(\alpha)=\psi_{-\beta_1,-\beta_2}
 (\mathbb{E}|\triangle_{0,1}X|^{\beta_1},\mathbb{E}|\triangle_{0,1}X|^{\beta_2})
\end{align}
where $h_{u,v}, \psi_{u,v}$ are defined by (\ref{eq.18.1}), (\ref{eq.19.1}), respectively.\\
  %\begin{equation}\label{eq.19}
 %\psi_{-\beta_1,-\beta_2}(W_n(\beta_1),W_n(\beta_2))-\psi_{-\beta_1,-\beta_2}(\mathbb{E}|\triangle_{0,1}X|^{\beta_1},\mathbb{E}|\triangle_{0,1}X|^{\beta_2})
 %=\large{O}_{\mathbb{P}}(b_n).
%\end{equation}
Indeed, since $\{X_t, t\in\mathbb{R}\}$ is a $H$-sssi $S\alpha S-$stable process, there exists a constant $\sigma>0$ such that
$\triangle_{0,1}X=\sigma Y$,
where $Y$ is the standard $H$-sssi, $S\alpha S$ random variable.
For $\beta_1,\beta_2\in\mathbb{R},-1/2<\beta_1,\beta_2<0$, from Lemma \ref{lem.01}, we have
\begin{align*}
\mathbb{E}|\triangle_{0,1}X|^{\beta_1} &=\sigma^{\beta_1} \mathbb{E}|Y|^{\beta_1}
=\sigma^{\beta_1}\frac{2^{\beta_1}\Gamma(\frac{\beta_1+1}{2})
\Gamma(1-\frac{\beta_1}{\alpha})}{\sqrt{\pi}\Gamma(1-\frac{\beta_1}{2})}.
\end{align*}
Thus
\begin{align*}
\left(\mathbb{E}|\triangle_{0,1}X|^{\beta_1}\right)^{\beta_2} 
&=\sigma^{\beta_1\beta_2}\left(\frac{2^{\beta_1}\Gamma(\frac{\beta_1+1}{2})
\Gamma(1-\frac{\beta_1}{\alpha})}{\sqrt{\pi}\Gamma(1-\frac{\beta_1}{2})}\right)^{\beta_2}.
\end{align*}
Similarly, we also get
\begin{align*}
\left(\mathbb{E}|\triangle_{0,1}X|^{\beta_2}\right)^{\beta_1} 
&=\sigma^{\beta_1\beta_2}\left(\frac{2^{\beta_2}\Gamma(\frac{\beta_2+1}{2})
\Gamma(1-\frac{\beta_2}{\alpha})}{\sqrt{\pi}\Gamma(1-\frac{\beta_2}{2})}\right)^{\beta_1}.
\end{align*}
Moreover, from Lemma \ref{lem.9}, $\mathbb{E}|\triangle_{0,1}X|^\beta\neq 0$ for all $ -1/2<\beta<0$, then it induces
\begin{align*}
 \frac{\left(\mathbb{E}|\triangle_{0,1}X|^{\beta_1}\right)^{\beta_2}}
 {\left(\mathbb{E}|\triangle_{0,1}X|^{\beta_2}\right)^{\beta_1}}&=
 \frac{\pi^{\frac{\beta_1-\beta_2}{2}}\Gamma^{\beta_1}(1-\frac{\beta_2}{2})
 \Gamma^{\beta_2}(\frac{\beta_1+1}{2})\Gamma^{\beta_2}(1-\frac{\beta_1}{\alpha})}
 {\Gamma^{\beta_2}(1-\frac{\beta_1}{2})
 \Gamma^{\beta_1}(\frac{\beta_2+1}{2})\Gamma^{\beta_1}(1-\frac{\beta_2}{\alpha})}.
\end{align*}
Taking the natural logarithm, we have
\begin{align*}
\beta_2\ln(\mathbb{E}|\triangle_{0,1}X|^{\beta_1})-\beta_1\ln(\mathbb{E}|\triangle_{0,1}X|^{\beta_2})=
&\frac{\beta_1-\beta_2}{2}\ln (\pi)+\beta_1 \ln \left( \Gamma(1-\frac{\beta_2}{2})\right)+\\
&\beta_2 \ln \left(\Gamma(\frac{\beta_1+1}{2})\right)+\beta_2\ln \left(\Gamma(1-\frac{\beta_1}{\alpha})\right)\\
&-\beta_2 \ln \left( \Gamma(1-\frac{\beta_1}{2})\right)-\beta_1 \ln \left(\Gamma(\frac{\beta_2+1}{2})\right)
-\beta_1\ln \left(\Gamma(1-\frac{\beta_2}{\alpha})\right).
\end{align*}
It follows that
\begin{align*}
\beta_2\ln \left(\Gamma(1-\frac{\beta_1}{\alpha})\right)- 
 \beta_1\ln \left(\Gamma(1-\frac{\beta_2}{\alpha})\right)=&
 \beta_2\ln(\mathbb{E}|\triangle_{0,1}X|^{\beta_1})-\beta_1\ln(\mathbb{E}|\triangle_{0,1}X|^{\beta_2})\\
 &+\frac{\beta_2-\beta_1}{2}\ln (\pi)-
 \beta_1 \ln \left( \Gamma(1-\frac{\beta_2}{2})\right)-\beta_2 \ln \left(\Gamma(\frac{\beta_1+1}{2})\right)\\
 &+\beta_2 \ln \left(\Gamma(1-\frac{\beta_1}{2})\right)
 +\beta_1 \ln \left(\Gamma(\frac{\beta_2+1}{2})\right)
\end{align*}
or 
$
h_{-\beta_1,-\beta_2}(\alpha)=\psi_{-\beta_1,-\beta_2}
 (\mathbb{E}|\triangle_{0,1}X|^{\beta_1},\mathbb{E}|\triangle_{0,1}X|^{\beta_2})
$.\\
From the following lemma, we can deduce that $h_{u,v}$ is a strictly increasing function on $(0,+\infty)$ and $\lim\limits_{x\rightarrow+\infty}h_{u,v}(x)=0, \lim\limits_{x\rightarrow 0}h_{u,v}(x)=-\infty $. Moreover, there exists an inverse function $$h_{u,v}^{-1}: (-\infty,0)\rightarrow (0,+\infty)$$ which is continuous and on $(-\infty, 0)$.  
 \begin{lem}\label{lem.13}
 Let $0<v<u$ and $g_{u,v}: (0,+\infty)\rightarrow \mathbb{R}$ be a function defined by  %(\ref{eq.17.1}):
  %\begin{align*}
 % g_{u,v}:& (0,+\infty)\rightarrow \mathbb{R}\\
 % & x\mapsto
 $$ g_{u,v}(x)=u\ln \left(\Gamma(1+vx)\right)- 
 v\ln \left(\Gamma(1+ux)\right).$$
%\end{align*}
%where $u,v$ are parameters such that $u,v\in \mathbb{R}, 0< v<u$. 
Then
   $ g_{u,v}$ is a strictly decreasing function on $(0,+\infty)$ and 
  $$
   \lim_{x\rightarrow 0}g_{u,v}(x)=0,\lim_{x\rightarrow +\infty}g_{u,v}(x)=-\infty. 
  $$
  \end{lem} 
   \begin{proof}%[\bf{Proof of Lemma \ref{lem.13}}]
 We have
  \begin{align*}
   g'_{u,v}(x)&=uv\frac{\Gamma'(1+vx)}{\Gamma(1+vx)}
   -uv\frac{\Gamma'(1+ux)}{\Gamma(1+ux)}=uv\left( \frac{\Gamma'(1+vx)}{\Gamma(1+vx)}
   -\frac{\Gamma'(1+ux)}{\Gamma(1+ux)} \right).
  \end{align*}
Following Bohn-Mollerup's theorem, $\Gamma$ is a log-convex function. Let $k(y)$ be defined by
$
k(y)=\ln \Gamma(y).
$
Then $k''(y)\geq 0$ for all $y> 0$. It follows that 
$\psi(y):=k'(y)=\frac{\Gamma'(y)}{\Gamma(y)}$ is an increasing function.\\
Since $\Gamma(x+1)=x\Gamma(x)$, we have
$\Gamma'(x+1)=\Gamma(x)+x\Gamma'(x)$. We obtain
\begin{align*}
\psi(x+1)&=\frac{\Gamma'(x+1)}{\Gamma(x+1)}=\frac{\Gamma(x)+x\Gamma'(x)}{x\Gamma(x)}
=\frac{1}{x}+\psi(x).
\end{align*}
We will prove that $\psi$ increases strictly.\\
Assume that there exist $x_0,y_0$ such that $0<x_0<y_0$ and $\psi(x_0)=\psi(y_0)$, then
$$\psi(x_0+1)-\psi(y_0+1)=\frac{1}{x_0}-\frac{1}{y_0}=\frac{y_0-x_0}{x_0y_0}>0.$$
However, $x_0+1<y_0+1$, then $\psi(x_0+1)\leq\psi(y_0+1)$ but this could not happen. 
Thus $\psi$ is a strictly increasing function.\\
We also have
$1<1+vx<1+ux$, so
$
\frac{\Gamma'(1+vx)}{\Gamma(1+vx)}
   -\frac{\Gamma'(1+ux)}{\Gamma(1+ux)}< 0.
$\\
It induces that $g'_{u,v}(x)< 0$ for all $0<v<u$ and $x>0$.
This proves that $g_{u,v}(x)$ is a strictly decreasing function.\\
It is clear that  $\lim\limits_{x\rightarrow 0}g_{u,v}(x)=0$. Now we need to prove that 
$\lim\limits_{x\rightarrow +\infty}g_{u,v}(x)=-\infty$.\\
Applying Stirling's formula, we have
$$\ln\Gamma(1+z)=\ln(z\Gamma(z))=\ln z+\ln\Gamma(z)=(z+\frac{1}{2})\ln z-z+\frac{1}{2}\ln(2\pi)+O(z^{-1})$$
as $z\rightarrow +\infty$. Then
\begin{align*}
 g_{u,v}(x)&=u\left( (vx+\frac{1}{2})\ln (vx)-vx+\frac{1}{2}\ln(2\pi)+O((vx)^{-1})\right)\\
 &-v\left( (ux+\frac{1}{2})\ln (ux)-ux+\frac{1}{2}\ln(2\pi)+O((ux)^{-1})\right)\\
 &=uv(\ln v-\ln u)x+\frac{u-v}{2}\ln x+\frac{u\ln v-v\ln u}{2}+\frac{(u-v)\ln (2\pi) }{2}+O(\frac{u}{vx})-
 O(\frac{v}{ux}) 
\end{align*}
as $x\rightarrow +\infty$. \\
Since $0<v<u$ we deduce $uv(\ln v-\ln u)<0$. Moreover 
$\lim\limits_{x\rightarrow+\infty}\frac{\ln x}{x}=0,$
it follows that
$$\lim\limits_{x\rightarrow +\infty}g_{u,v}(x)=-\infty.$$
\end{proof}
Then since $\alpha\in(0,2]$, we obtain that $
h_{-\beta_1,-\beta_2}(\alpha)=\psi_{-\beta_1,-\beta_2}
 (\mathbb{E}|\triangle_{0,1}X|^{\beta_1},\mathbb{E}|\triangle_{0,1}X|^{\beta_2})<0
$. On the other hand
$$\psi_{-\beta_1,-\beta_2}(W_n(\beta_1),W_n(\beta_2))=\psi_{-\beta_1,-\beta_2}(V_n(\beta_1),V_n(\beta_2)).$$
We deduce that
\begin{align*}
 \hat{\alpha}_n-\alpha&=\varphi_{-\beta_1,-\beta_2}\left(\psi_{-\beta_1,-\beta_2}(W_n(\beta_1),W_n(\beta_2))\right)
 -h^{-1}_{-\beta_1,-\beta_2}\left(h_{-\beta_1,-\beta_2}(\alpha)\right)\\
 &=\varphi_{-\beta_1,-\beta_2}\left(\psi_{-\beta_1,-\beta_2}(W_n(\beta_1),W_n(\beta_2))\right)
 -\varphi_{-\beta_1,-\beta_2}\left(h_{-\beta_1,-\beta_2}(\alpha) \right)\\
 &=\varphi_{-\beta_1,-\beta_2}\left(\psi_{-\beta_1,-\beta_2}(W_n(\beta_1),W_n(\beta_2))\right)
 -\varphi_{-\beta_1,-\beta_2}(\psi_{-\beta_1,-\beta_2}(\mathbb{E}|\triangle_{0,1}X|^{\beta_1},\mathbb{E}|\triangle_{0,1}X|^{\beta_2})).
\end{align*}
Moreover, $\varphi_{-\beta_1,-\beta_2}\circ \psi_{-\beta_1,-\beta_2}$ is continuous and differentiable at 
\begin{equation}\label{eq.20}
 x_0=
(\mathbb{E}|\triangle_{0,1}X|^{\beta_1},\mathbb{E}|\triangle_{0,1}X|^{\beta_2}).
\end{equation}
 Combining with (\ref{eq.10.1}), (\ref{eq.10.2}), we apply Lemma \ref{lem.B.3}, and get 
 $$\widehat{\alpha}_n-\alpha=O_\mathbb{P}(b_n).$$ 
 It also induces that $\lim\limits_{n\rightarrow +\infty}\widehat{\alpha}_n\stackrel{(\mathbb{P})}{=}\alpha$.
\end{proof}
\subsection{Proofs related to Section \ref{section.3}}\label{subsection.4.3}
Now we are in position to prove Theorems related to examples presented in Section \ref{section.3}.
\begin{proof}[\bf{Proof of Theorem \ref{thm.6}}]
%\begin{proof}[\bf{Proof of Theorem \ref{thm.6}}]
a) From Lemma \ref{lem.11} in Appendix, the assumption (\ref{eq.9.2}) is satisfied. Then following Theorem \ref{thm.5}, we have 
$$\hat{H}_n-H=O_\mathbb{P}(n^{-1/2}), \hat{\alpha}_n-2= O_\mathbb{P}(n^{-1/2}).$$
b) We now prove the asymptotic normality for the estimators of $H$ and $\alpha$. To prove $\sqrt{n}(\hat{H}_n-H)$ converges to a normal distribution as $n\rightarrow +\infty$, we will first prove that 
\begin{align}\label{eq.12.1}
\sqrt{n}\left((W_{n}(\beta),W_{n/2}(\beta))-
(\mathbb{E}|\triangle_{0,1}X|^\beta,\mathbb{E}|\triangle_{0,1}X|^\beta) \right)&
\xrightarrow{(d)}\mathcal{N}_2(0,\varGamma_1)
\end{align}
as $n\rightarrow +\infty$, where $\varGamma_1$ is defined by (\ref{eq.11.2.3}). 
Then, we need to prove that for all $a,b\in\mathbb{R}, ab\neq 0$, 
\begin{equation}\label{eq.12.2}
V_n:= a\sqrt{n}(W_{n}(\beta)-\mathbb{E}|\triangle_{0,1}X|^{\beta})+
 b\sqrt{n}(W_{n/2}(\beta)-\mathbb{E}|\triangle_{0,1}X|^{\beta})
\end{equation}
converges to $G\sim\mathcal{N}_1(0,\sigma^2)$ as $n\rightarrow +\infty$, where 
\begin{equation}\label{eq.12.3}
\sigma^2=(a^2+2b^2)\sum\limits_{q\geq d}q!f^2_q\sum\limits_{r\in\mathbb{Z}}\rho^q(r)+
2ab\sum\limits_{q\geq d}q!f^2_q\sum\limits_{r\in\mathbb{Z}}\rho^q_1(r),
\end{equation}
$f_q$s, $\rho, \rho_1$ are defined by(\ref{eq.11.1}), (\ref{eq.11.2.1}), (\ref{eq.11.2.2}), respectively.
 Since $\{X_t\}_{t\geq 0}$ is a $H$-sssi process, for all $n\in \mathbb{N}^*$, we get
  \begin{align*}
   &\left( \triangle_{0,n}X,\triangle_{1,n}X,\ldots, \triangle_{n-K,n}X, 
   \triangle_{0,n/2}X,\triangle_{1,n/2}X,\ldots, \triangle_{n/2-K,n/2}X
   \right)\\
  & \stackrel{(d)}{=}
  \frac{1}{(n/2)^H}\left( \triangle_{0,2}X,\triangle_{1,2}X,\ldots, \triangle_{n-K,2}X, 
   \triangle_{0,1}X,\triangle_{1,1}X,\ldots, \triangle_{n/2-K,1}X\right).
  \end{align*}
  Moreover $var \triangle_{k,2}=\frac{var \triangle_{0,1}X}{2^{2H}}, 
var \triangle_{k,1}X=var \triangle_{0,1}X$. It follows that
\begin{align*}
 &\sqrt{n}\left((W_{n}(\beta),W_{n/2}(\beta))-(\mathbb{E}|\triangle_{0,1}X|^\beta,\mathbb{E}|\triangle_{0,1}X|^\beta) \right)\\
  &\stackrel{(d)}{=}
 \sqrt{n}\left(\frac{(n)^{\beta H}}{n-K+1}
 \sum\limits_{k=0}^{n-K}\frac{|\triangle_{k,2}|^{\beta}}{(n/2)^{\beta H}}-\mathbb{E}|\triangle_{0,1}X|^{\beta},
 \frac{(n/2)^{\beta H}}{n/2-K+1}
 \sum\limits_{k=0}^{n/2-K}\frac{|\triangle_{k,1}|^{\beta}}{(n/2)^{\beta H}}-\mathbb{E}|\triangle_{0,1}X|^{\beta}\right)\\
 &=
 \sqrt{n}\left(\frac{2^{\beta H}}{n-K+1}
 \sum\limits_{k=0}^{n-K}|\triangle_{k,2}X|^{\beta}-\mathbb{E}|\triangle_{0,1}X|^{\beta},
 \frac{1}{n/2-K+1}
 \sum\limits_{k=0}^{n/2-K}|\triangle_{k,1}X|^{\beta}-\mathbb{E}|\triangle_{0,1}X|^{\beta}\right)\\
 &=(var \triangle_{0,1}X)^{\beta/2} \left(\frac{\sqrt{n}}{n-K+1}
 \sum\limits_{k=0}^{n-K}(|Y_k|^{\beta}
 -\mathbb{E}|Z_0|^{\beta}),
 \frac{\sqrt{n}}{n/2-K+1}
 \sum\limits_{l=0}^{n/2-K}(|Z_l|^{\beta}
 -\mathbb{E}|Z_0|^{\beta})\right)\\
 &=\left(\frac{\sqrt{n}}{n-K+1}
 \sum\limits_{k=0}^{n-K}f_{\beta}(Y_k),
 \frac{\sqrt{n}}{n/2-K+1}
 \sum\limits_{l=0}^{n/2-K}f_{\beta}(Z_l)\right)
 \end{align*}
 where $Y_k=\frac{\triangle_{k,2}X}{\sqrt{var \triangle_{k,2}X}},
Z_l=\frac{\triangle_{l,1}X}{\sqrt{var \triangle_{l,1}X}}$ and $f_\beta$ is defined by (\ref{eq.11.0}).\\ 
 We obtain that 
$Y_k\sim\mathcal{N}_1(0,1),Z_l\sim\mathcal{N}_1(0,1)$, and
\begin{align*}
 \mathbb{E} Y_kY_{k'}&=\frac{\mathbb{E}(\triangle_{k,2}X\triangle_{k',2}X)}{\frac{var\triangle_{0,1}X}{2^{2H}}}=\frac{\sum\limits_{p,p'=0}^Ka_pa_{p'}|k-k'+p-p'|^{2H}}{\sum\limits_{p,p'=0}^Ka_pa_{p'}|p-p'|^{2H}}.%=\rho(k-k'),
\end{align*}
\begin{equation*}
 \mathbb{E} Y_kZ_l=\frac{\mathbb{E}(\triangle_{k,2}X\triangle_{l,1}X)}{\frac{var\triangle_{0,1}X}{2^{H}}}\\
 =\frac{\sum\limits_{p,p'=0}^Ka_pa_{p'}|k-2l+p-2p'|^{2H}}{2^H\sum\limits_{p,p'=0}^Ka_pa_{p'}|p-p'|^{2H}},%=\rho_1(k-2l).
\end{equation*}
\begin{equation*}
 \mathbb{E} Z_lZ_{l'}=\frac{\mathbb{E}(\triangle_{l,1}X\triangle_{l',1}X)}{\frac{var\triangle_{0,1}X}{2^{H}}}\\
 =\frac{\sum\limits_{p,p'=0}^Ka_pa_{p'}|k-l+p-p'|^{2H}}{\sum\limits_{p,p'=0}^Ka_pa_{p'}|p-p'|^{2H}},%=\rho_1(k-2l).
\end{equation*}
Then $\mathbb{E} Y_kY_{k'}=\rho(k-k'), \mathbb{E} Z_lZ_{l'}=\rho (l-l') $ and $\mathbb{E} Y_kZ_l=\rho_1(k-2l)$ where $\rho, \rho_1$ are defined by (\ref{eq.11.2.1}), (\ref{eq.11.2.2}), respectively.  %Similarly, we also get $\mathbb{E} Z_lZ_{l'}=\rho(l-l')$.\\
As in the proof of Lemma \ref{lem.11} in Appendix, we can prove that for $r$ big enough
\begin{equation}\label{eq.12.3.1}
 \rho(r)|\leq C |r|^{2H-3},|\rho_1(r)|\leq C |r|^{2H-3}.
\end{equation}
We then mimic the proof of Theorem 7.2.4 in \cite{Nourdin2012} to get $V_n\xrightarrow{(d)}\mathcal{N}_1(0,\sigma^2)$ as $n\rightarrow +\infty$, it follows (\ref{eq.12.1}). On the other hand, we have
 \begin{align*}
  \sqrt{n}(\hat{H}_{n}-H)&=\sqrt{n}\frac{1}{\beta}\log_2\frac{W_{n/2}(\beta)}{W_{n}(\beta)}=\sqrt{n}\left(\phi(W_{n}(\beta),W_{n/2}(\beta))-
  \phi(\mathbb{E}|\triangle_{0,1}X|^\beta,\mathbb{E}|\triangle_{0,1}X|^\beta)\right).
 \end{align*}
 where $\phi$ is defined as in (\ref{eq.11.3.1}).\\
  Since $\phi$ is differentiable at $(x_0,y_0)=(\mathbb{E}|\triangle_{0,1}X|^\beta,\mathbb{E}|\triangle_{0,1}X|^\beta)$,
we can apply Theorem 3.1 in \cite{Dacunha1983} to get 
$$\sqrt{n}(\hat{H}_{n}-H)\xrightarrow{(d)}\mathcal{N}_1(0,\Xi_1)$$
as $n\rightarrow +\infty$, where $\Xi_1$ is defined by (\ref{eq.11.2.4}).\\ 
Now we prove central limit theorem for the estimation of $\alpha$. We will prove that 
\begin{equation}\label{eq.27.1}
 \left( \sqrt{n}(W_n(\beta_1)-\mathbb{E}|\triangle_{0,1}X|^{\beta_1}),
 \sqrt{n}(W_n(\beta_2)-\mathbb{E}|\triangle_{0,1}X|^{\beta_2})) \right)\xrightarrow{(d)}\mathcal {N}_2(0,\varGamma_2)
\end{equation}
as $n\rightarrow +\infty$, with $\varGamma_2$ defined by (\ref{eq.23}).\\
Since $\{X_t\}_{t\in \mathbb{R}}$ is a $H$-sssi process, we have
 \begin{equation}\label{eq.28}
  \left( \triangle_{0,n}X,\ldots, \triangle_{n-K,n}X \right)\stackrel{(d)}{=}
  \frac{1}{n^H}\left( \triangle_{0,1}X,\ldots, \triangle_{n-K,1}X \right).
 \end{equation}
On the other hand, $var \triangle_{k,1}X= var \triangle_{0,1}X$. Then we can write
\begin{align*}
 &\left( \sqrt{n}(W_n(\beta_1)-\mathbb{E}|\triangle_{0,n}X|^{\beta_1}),
 \sqrt{n}(W_n(\beta_2)-\mathbb{E}|\triangle_{0,n}X|^{\beta_2})) \right)\\
 &\stackrel{(d)}{=} \left( \sqrt{n}\left(\frac{n^{\beta_1 H}}{n-K+1}
 \sum_{k=0}^{n-K}\frac{|\triangle_{k,1}X|^{\beta_1}}{n^{\beta_1 H}}-\mathbb{E}|\triangle_{0,1}X|^{\beta_1}\right), 
 \sqrt{n}\left(\frac{n^{\beta_2 H}}{n-K+1}
 \sum_{k=0}^{n-K}\frac{|\triangle_{k,1}X|^{\beta_2}}{n^{\beta_2 H}}-\mathbb{E}|\triangle_{0,1}X|^{\beta_2}\right)\right)
 \\
  &=\left( \sqrt{n}\left(\frac{1}{n-K+1}
 \sum_{k=0}^{n-K}|\triangle_{k,1}X|^{\beta_1}-\mathbb{E}|\triangle_{0,1}X|^{\beta_1}\right), 
 \sqrt{n}\left(\frac{1}{n-K+1}
 \sum_{k=0}^{n-K}|\triangle_{k,1}X|^{\beta_2}-\mathbb{E}|\triangle_{0,1}X|^{\beta_2}\right)\right) \\
&=\sqrt{\frac{n}{n-K+1}}
\left( \frac{1}{\sqrt{n-K+1}}\sum_{k=0}^{n-K}f_{\beta_1}(Z_k), \frac{1}{\sqrt{n-K+1}}\sum_{k=0}^{n-K}f_{\beta_2}(Z_k)\right).
\end{align*}
where $f_{\beta_1}$ and $f_{\beta_2}$ are defined as in (\ref{eq.11.0}) and $Z_k=\frac{\triangle_{k,1}X}{\sqrt{var \triangle_{k,1}X}}, Z_k\sim \mathcal{N}_1(0,1)$.\\
We have $\mathbb{E}f_{\beta_1}(Z_0)=\mathbb{E}f_{\beta_2}(Z_0)=0, \mathbb{E}f_{\beta_1}^2(Z_0)<+\infty, \mathbb{E}f_{\beta_2}^2(Z_0)<+\infty$ and $\mathbb{E}Z_kZ_l=\rho(k-l)$ where $\rho$ is defined by (\ref{eq.11.2.1}).\\
We mimic the proof of Theorem 7.2.4 of \cite{Nourdin2012} to obtain that 
$$
a\sqrt{n}\left( W_n(\beta_1)-\mathbb{E}|\triangle_{0,1}|^{\beta_1} \right)+b\sqrt{n}\left( W_n(\beta_2)-\mathbb{E}|\triangle_{0,1}|^{\beta_2} \right)
$$
converges to $\mathcal {N}_1(0,\sigma^2)$ as $n\rightarrow +\infty$ for all $a,b\in\mathbb{R},ab\neq 0$, where
$\sigma^2=\sum\limits_{q=d}^{+\infty}q!(ah_q+bg_q)^2\sum\limits_{r\in\mathbb{Z}}\rho(r)^q.$
Here $ \rho, h_q, g_q$ are defined by (\ref{eq.11.2.1}), (\ref{eq.21.1}) respectively. This proves (\ref{eq.27.1}).\\
The function
$\varphi_{-\beta_1,-\beta_2}\circ \psi_{-\beta_1,-\beta_2}: \mathbb{R}^{+}\times\mathbb{R}^{+}\rightarrow[0,+\infty)$ 
is differentiable at 
$$(x_1,y_1)=(\mathbb{E}|\triangle_{0,1}X|^{\beta_1},\mathbb{E}|\triangle_{0,1}X|^{\beta_2}),$$
 where 
$\psi_{u,v}, \varphi_{u,v}$ are defined by (\ref{eq.19.1}), (\ref{eq.20.1}) respectively.\\
We can therefore apply Theorem 3.1 of \cite{Dacunha1983} to get the conclusion.
\end{proof}
%\end{proof}
\begin{proof}[\bf{Proof of Theorem \ref{thm.7}}]
a) We will check the assumption (\ref{eq.9.2}).\\
 Let $0\leq l<k, k-l\geq K$, then
 $
 \triangle_{k,1}X=\sum\limits_{p=0}^Ka_p\left[X(k+p)-X(k)\right]
 $
 and
 $
 \triangle_{l,1}X=\sum\limits_{p'=0}^Ka_{p'}\left[X(l+p')-X(l)\right].
 $\\
 By the fact that $\{X_t\}_{t\geq 0}$ has independent increments, we obtain that 
 %\begin{align*}
 $X(l+p')-X(l),X(k+p)-X(k)$
 %\end{align*}
are independent for all $p,p'=0,\ldots, K$ since $0\leq l\leq l+p'\leq k\leq k+p$.\\
It follows that $\triangle_{k,1}X$ and $\triangle_{0,1}X$ are independent for $|k|\geq K$. Thus
$ cov(\mid\triangle_{k,1}X\mid^\beta,\mid\triangle_{0,1}X\mid^\beta)=0.
 $
 We deduce that
$$\frac{1}{n}\sum_{k\in \mathbb{Z},|k|\leq n}
 \mid cov(\mid\triangle_{k,1}X\mid^\beta,\mid\triangle_{0,1}X\mid^\beta)\mid=
 \frac{1}{n}\sum_{k\in \mathbb{Z},|k|\leq K}
 \mid cov(\mid\triangle_{k,1}X\mid^\beta,\mid\triangle_{0,1}X\mid^\beta)\mid=\frac{C}{n}
 $$
 where $C$ is a positive constant. We thus get (\ref{eq.9.2}) with $b_n=n^{-1/2}$ and it follows that 
 $\hat{H}_n-H=O_\mathbb{P}(n^{-1/2}),\hat{\alpha}_n-\alpha=O_\mathbb{P}(n^{-1/2})$.\\
 b)  To prove the asymptotic normality for the estimator of $H$, we first prove that for all $n\in\mathbb{N}, n>2K$, 
 \begin{equation*}
\sqrt{n}\left((W_{n},W_{n/2})-
(\mathbb{E}|\triangle_{0,1}X|^\beta,\mathbb{E}|\triangle_{0,1}X|^\beta) \right)
\end{equation*}
converges in distribution to a normal distribution
as $n\rightarrow +\infty$.\\
 Since $\{X_t\}_{t\geq 0}$ is a $H$-sssi process, one has
  \begin{align*}
   &\left( \triangle_{0,n}X,\triangle_{1,n}X,\ldots, \triangle_{n-K,n}X, 
   \triangle_{0,n/2}X,\triangle_{1,n/2}X,\ldots, \triangle_{n/2-K,n/2}X
   \right)\\
  & \stackrel{(d)}{=}
  \frac{1}{(n/2)^H}\left( \triangle_{0,2}X,\triangle_{1,2}X,\ldots, \triangle_{n-K,2}X, 
   \triangle_{0,1}X,\triangle_{1,1}X,\ldots, \triangle_{n/2-K,1}X\right).
  \end{align*}
  Moreover $$\mathbb{E}|\triangle_{k,2}|^{\beta}=\frac{\mathbb{E}|\triangle_{0,1}X|^{\beta}}{2^{\beta H}}, 
\mathbb{E}|\triangle_{k,1}X|^{\beta}=\mathbb{E}|\triangle_{0,1}X|^{\beta}, $$
$$
var|\triangle_{k,2}|^{\beta}=\frac{var|\triangle_{0,1}X|^{\beta}}{2^{2\beta H}}, 
var|\triangle_{k,1}X|^{\beta}=var|\triangle_{0,1}X|^{\beta}.$$
It follows that 
\begin{eqnarray*}
&\sqrt{n}\left(\left(W_{n},W_{n/2}\right)-
\left(\mathbb{E}|\triangle_{0,1}X|^\beta,\mathbb{E}|\triangle_{0,1}X|^\beta \right) \right)\\
&\stackrel{(d}{=} \sqrt{n}\left( \left(\frac{2^{\beta H}}{n-K +1}
\sum\limits_{p=0}^{n-K}|\triangle_{p,2}X)|^{\beta},
\frac{1}{n/2-K +1}
\sum\limits_{p=0}^{n/2-K}|\triangle_{p,1}X)|^{\beta}\right)
-\left(\mathbb{E}|\triangle_{0,1}X|^\beta,\mathbb{E}|\triangle_{0,1}X|^\beta \right)
\right) \\
&=\sqrt{n}\left(\frac{2^{\beta H}}{n-K +1}
\sum\limits_{p=0}^{n-K}\left(\triangle_{p,2}X)|^{\beta}-\mathbb{E}|\frac{\triangle_{0,1}X}{2^H}|^\beta\right),
\frac{1}{n/2-K +1}
\sum\limits_{p=0}^{n/2-K}\left(|\triangle_{p,1}X)|^{\beta}-\mathbb{E}|\triangle_{0,1}X|^\beta\right)
\right).
\end{eqnarray*}
 Now we need to prove that for all $a,b\in\mathbb{R}, ab\neq 0$, 
\begin{align*}
 S_n:=&a\sqrt{n}\left(\frac{2^{\beta H}}{n-K +1}
\sum\limits_{p=0}^{n-K}\left(|\triangle_{p,2}X)|^{\beta}-\mathbb{E}|\frac{\triangle_{0,1}X}{2^H}|^\beta\right)\right)\\
&+
 b\sqrt{n}\left(\frac{1}{n/2-K +1}
\sum\limits_{p=0}^{n/2-K}\left(|\triangle_{p,1}X)|^{\beta}-\mathbb{E}|\triangle_{0,1}X|^\beta\right)
\right)
\end{align*}
converges to a normal distribution when $n\rightarrow +\infty$. Let 
\begin{equation}\label{eq.28.2}
 Z_p:=\frac{2^{\beta H}a}{2}
 \left(|\triangle_{2p,2}X|^{\beta}+ 
 |\triangle_{2p+1,2}X|^{\beta}  \right)
 +b|\triangle_{p,1}X|^{\beta}.
\end{equation}
%If $p-p'\geq K$, then 
 It follows that 
 \[
 S_n=\sqrt{\frac{n}{n/2-K+1}}\left(\frac{1}{\sqrt{n/2-K+1}}\sum\limits_{p=0}^{n/2-K}(Z_p-\mathbb{E}Z_p\right)+U_n,
\]
where 
\begin{align*}
U_n=&\frac{2^{\beta H}a\sqrt{n}}{n-K+1}
\left(\frac{1-K}{n-2K+2}\sum\limits_{p=0}^{n-2K+1}
\left(|\triangle_{p,2}X|^{\beta}-\mathbb{E}|\frac{\triangle_{0,1}X}{2^H}|^\beta\right)+
\sum\limits_{p=n-2K+2}^{n-K}
\left(|\triangle_{p,2}X|^{\beta}-\mathbb{E}|\frac{\triangle_{0,1}X}{2^H}|^\beta\right)
\right)\\
%&\times\frac{2^{\beta H}a\sqrt{n}}{n-K+1}\\
=&\frac{2^{\beta H}a\sqrt{n}}{n-K+1}
\left( \frac{1-K}{n-2K+2}\sum\limits_{p=0}^{n-2K+1}Y_p+
\sum\limits_{p=n-2K+2}^{n-K}Y_p\right),
Y_p=|\triangle_{p,2}X|^{\beta}-\mathbb{E}|\frac{\triangle_{0,1}X}{2^H}|^\beta.
\end{align*}
Since $\sum\limits_{k=0}^K a_k=0$, one can write
\begin{align*}
Z_p=&\frac{2^{\beta H}a}{2}
 \left(\left|\sum\limits_{k=0}^Ka_k\left(X(\frac{k+2p}{2})-X(p)\right)\right|^{\beta}
 + \left|\sum\limits_{k=0}^Ka_k\left(X(\frac{k+2p+1}{2})-X(p)\right)\right|^{\beta}
\right)\\
&+b\left|\sum\limits_{k=0}^Ka_k\left(X(k+p)-X(p)\right)\right|^{\beta}.
\end{align*}
  If $p-p'> K-1$, since $X$ has independent increments and
 \begin{align*}
  & 0\leq p'\leq \frac{k+2p'}{2}\leq p \leq \min\{\frac{k+2p}{2},\frac{k+2p+1}{2},k+p\},\\
  & 0\leq p'\leq \frac{k+2p'+1}{2}\leq p \leq \min\{\frac{k+2p}{2},\frac{k+2p+1}{2},k+p\},\\
  & 0\leq p'\leq k+p'\leq p \leq \min\{\frac{k+2p}{2},\frac{k+2p+1}{2},k+p\},
\end{align*}
for all $k=0,\ldots, K$, it follows that
$Z_p,Z_{p'}$ are independent.
 It induces that $\{Z_p\}_{p\in\mathbb{N}}$ is a $(K-1)$-dependent sequence 
of random variables.
For $l\in\mathbb{R}$ fixed, we also have
\begin{align*}
 Z_{p+l}&=\frac{2^{\beta H}a}{2}
 \left(\left|\sum\limits_{k=0}^Ka_k\left(X(\frac{k+2(p+l)}{2})-X(l)\right)\right|^{\beta}
 +\left|\sum\limits_{k=0}^Ka_k\left(X(\frac{k+2(p+l)+1}{2})-X(l)\right)\right|^{\beta}
 \right)\\
 &+b\left|\sum\limits_{k=0}^Ka_k\left(X(k+p+l)-X(l)\right)\right|^{\beta}.
   \end{align*}
    On the other hand, since $X$ has stationary increments and $X(0)=0$ almost surely, we have
   $$
   \left(X(t+l)-X(l) 
   \right)_{t\in\mathbb{R}}\stackrel{(d)}{=}\left(X(t) 
   \right)_{t\in\mathbb{R}}.
   $$
   Then
$
(Z_{p+l}, p\in \mathbb{R})\stackrel{(d)}{=}(Z_p, p\in \mathbb{R})
$
or in another way, $(Z_p, p\in \mathbb{R})$ is stationary.\\
It follows that $\{Z_p\}_{p\in\mathbb{N}}$ is a stationary $(K-1)$-dependent sequence of random variables. From Theorem 2.8.1 in \cite{Lehmann1999}, we get
$$
\sqrt{\frac{n}{n/2-K+1}}\left(\frac{1}{\sqrt{n/2-K+1}}\sum\limits_{p=0}^{n/2-K}(Z_p-\mathbb{E}Z_p)\right)
$$
converges in distribution to a centered normal distribution with variance
\begin{align}\label{eq.28.2.1}
\sigma^2&=2(var Z_0+2\sum\limits_{k=1}^{K-1}cov(Z_0,Z_k))=a^2\sigma_1^2+b^2\sigma_2^2+2ab\sigma_{1,2}
\end{align}
where $\sigma_1^2,\sigma_2^2,\sigma_{1,2}$ are defined as in (\ref{eq.31}), (\ref{eq.32}), (\ref{eq.33}).
We also have $\mathbb{E}Y_p=0,\mathbb{E}U_n=0, Y_p\stackrel{(d)}{=}Y_0$ and $\mathbb{E}Y_p^2=\mathbb{E}Y_0^2
$ for all $p$.
Thus
\begin{align*}
 \mathbb{E}U_n^2&=\frac{2^{2\beta H}a^2n}{(n-K+1)^2}
 \mathbb{E}\left( \frac{1-K}{n-2K+2)}\sum\limits_{p=0}^{n-2K+1}Y_p+
\sum\limits_{p=n-2K+2}^{n-K}Y_p\right)^2\\
&\leq \frac{2^{2\beta H+1}a^2n}{(n-K+1)^2}
\left(\frac{(1-K)^2}{(n-2K+2)^2}\mathbb{E}\left(\sum\limits_{p=0}^{n-2K+1}Y_p\right)^2+
\mathbb{E}\left(\sum\limits_{p=n-2K+2}^{n-K}Y_p\right)^2\right)\\
&\leq \frac{2^{2\beta H+1}a^2n}{(n-K+1)^2}
\left(
\frac{(1-K)^2}{(n-2K+2)^2}(n-2K+2)^2\mathbb{E}Y_0^2+(K-1)^2\mathbb{E}Y_0^2
\right)\\
&=\frac{2^{2\beta H+2}a^2(K-1)^2n}{(n-K+1)^2}\mathbb{E}Y_0^2=\frac{nC}{(n-K+1)^2}.
\end{align*}
It follows that $\mathbb{E}U_n^2$ converges to $0$ as $n\rightarrow +\infty$. 
Moreover $\mathbb{E}U_n=0$, using Chebyshev's inequality, we obtain that $U_n\xrightarrow{(\mathbb{P})}0$ as $n\rightarrow +\infty$.\\
Following Slutsky's theorem, as
$n\rightarrow +\infty$, $S_n$ converges in distribution to a centered normal distribution with variance $\sigma$ as in (\ref{eq.28.2}). \\
We deduce that 
$\sqrt{n}\left((W_{n},W_{n/2})-
(\mathbb{E}|\triangle_{0,1}X|^\beta,\mathbb{E}|\triangle_{0,1}X|^\beta) \right)\stackrel{(d)}{\rightarrow}\mathcal{N}_2(0,\varGamma_3)$
where $\varGamma_3$ is defined by (\ref{eq.34}).
Since
 \begin{align}\label{eq.28.4}
  \sqrt{n}(\hat{H}_{n}-H)&=\sqrt{n}\frac{1}{\beta}\log_2\frac{W_{n/2}}{W_{n}}=\sqrt{n}\left(\phi(W_{n}(\beta),W_{n/2}(\beta))-
  \phi(\mathbb{E}|\triangle_{0,1}X|^\beta,\mathbb{E}|\triangle_{0,1}X|^\beta)\right)
 \end{align}
  where $\phi$ is defined by (\ref{eq.11.3.1}). Applying Theorem 3.1 of \cite{Dacunha1983}, we get $\sqrt{n}(\hat{H}-H)\stackrel{(d)}{\rightarrow}\mathcal{N}_1(0,\Xi_2)$ with $\Xi_2$ defined by (\ref{eq.35}).\\
  We now prove the central limit theorem for the estimation of $\alpha$ in the case of $S\alpha S-$stable L\'evy motion.\\
 We need to prove that for all $n\in\mathbb{N}, n>K$, then $\sqrt{n}\left(( W_n(\beta_1),W_n(\beta_2))-(\mathbb{E}|\triangle_{0,1}X|^{\beta_1},\mathbb{E}|\triangle_{0,1}X|^{\beta_2})\right)$ converges in distribution to a normal distribution as $n\rightarrow +\infty$.\\
We consider
$$S_n=a\sqrt{n}\left(W_n(\beta_1)-\mathbb{E}|\triangle_{0,1}X|^{\beta_1}\right)+b\sqrt{n}\left(W_n(\beta_2)-\mathbb{E}|\triangle_{0,1}X|^{\beta_2}\right)$$
for all $a,b\in\mathbb{R},ab\neq 0$. Since $\{X_t,t\in\mathbb{R}\}$ is a $H$ self-similar process, we have
\begin{align*}
S_n&\stackrel{(d)}{=}\frac{\sqrt{n}}{n-K+1}\sum\limits_{k=0}^{n-K}\left( a(|\triangle_{k,1}X|^{\beta_1}-\mathbb{E}|\triangle_{k,1}X|^{\beta_1})+ b(|\triangle_{k,1}X|^{\beta_2}-\mathbb{E}|\triangle_{k,1}X|^{\beta_2})\right).\\
&=\frac{\sqrt{n}}{n-K+1}\sum\limits_{k=0}^{n-K}(Z_k-\mathbb{E}Z_k)
\end{align*}
where
\begin{equation}\label{eq.28.4.1}
Z_k=a|\triangle_{k,1}X|^{\beta_1}+b|\triangle_{k,1}X|^{\beta_2}.
\end{equation} 
Since $\{X_t,t\in\mathbb{R}\}$ has stationary increments, $\{Z_k,k\in\mathbb{N}\} $ is stationary.\\
 Moreover, if $k-k'> K-1$, since $\{X_t,t\in\mathbb{R}\}$ has independent increments, then $Z_k,Z_{k'}$ are independent. We obtain that $\{Z_k,k\in\mathbb{N}\} $ is a stationary $(K-1)$-dependent sequence of random variables. Then applying Theorem 2.8.1 of \cite{Lehmann1999}, as $n\rightarrow +\infty$, $S_n$ converges to a centered normal distribution with variance:
 \begin{equation}\label{eq.28.5}
 \sigma^2=var Z_0+2\sum\limits_{k=0}^{K-1}cov(Z_0,Z_k).
 \end{equation}
  We can write $\sigma^2$ in details
 \begin{align*}
\sigma^2&=a^2 \left(var|\triangle_{0,1}X|^{\beta_1}+2\sum\limits_{k=1}^{K-1}cov(|\triangle_{0,1}X|^{\beta_1},|\triangle_{k,1}X|^{\beta_1})\right)\\
&+b^2 \left(var|\triangle_{0,1}X|^{\beta_2}+2\sum\limits_{k=1}^{K-1}cov(|\triangle_{0,1}X|^{\beta_2},|\triangle_{k,1}X|^{\beta_2})\right)+2ab\\
&\times\left[cov(|\triangle_{0,1}X|^{\beta_1},|\triangle_{0,1}X|^{\beta_2})+\frac{1}{2}\sum\limits_{k=1}^{K-1} \left( cov(|\triangle_{0,1}X|^{\beta_1},|\triangle_{k,1}X|^{\beta_2})+cov(|\triangle_{0,1}X|^{\beta_2},|\triangle_{k,1}X|^{\beta_1}) \right) \right]\\
&=a^2\sigma_1^2+b^2\sigma_2^2+2ab\sigma_{1,2},
\end{align*}
where $\sigma_1^2,\sigma_2^2,\sigma_{1,2}$ are defined by (\ref{eq.27.0}), (\ref{eq.27}), (\ref{eq.28.1}) respectively. \\
It follows that 
$\sqrt{n}\left(( W_n(\beta_1),W_n(\beta_2))-(\mathbb{E}|\triangle_{0,1}X|^{\beta_1},\mathbb{E}|\triangle_{0,1}X|^{\beta_2})\right)\stackrel{(d)}{\rightarrow} \mathcal{N}_2(0,\Gamma_4)$, where $\varGamma_4$ defined by (\ref{eq.30}).\\
  The function
$\varphi_{-\beta_1,-\beta_2}\circ \psi_{-\beta_1,-\beta_2}: \mathbb{R}^{+}\times\mathbb{R}^{+}\rightarrow[0,+\infty)$ 
is differentiable at 
$$(x_1,y_1)=(\mathbb{E}|\triangle_{0,1}X|^{\beta_1},\mathbb{E}|\triangle_{0,1}X|^{\beta_2}),$$ where $\psi_{u,v}, \varphi_{u,v}$ are defined by (\ref{eq.19.1}), (\ref{eq.20.1}) respectively. Then we apply Theorem 3.1 of \cite{Dacunha1983} to get the conclusion. 
\end{proof}
\begin{proof}[\bf{Proof of Theorem \ref{thm.8}}]
Set 
$
f(t)=\sum_{k=0}^K a_k|k-t|^{H-1/\alpha}.
$
For all $k\in\mathbb{Z}$, one has 
\begin{align*}
\triangle_{k,1}X&%=\sum_{j=0}^Ka_jX(k+j)
=\int\limits_{\mathbb{R}}\sum_{j=0}^Ka_j(\mid k+j-s\mid^{H-1/\alpha}
-\mid s\mid^{H-1/\alpha})M(ds)
%&=\int\limits_{\mathbb{R}}\sum_{j=0}^Ka_j\mid k+j-s\mid^{H-1/%\alpha})M(ds), {\text{  (since $\sum_{j=0}^Ka_j=0$)}}\%\
%&=\int_{\mathbb{R}}\sum_{j=0}^Ka_j\mid j-(s-k)%\mid^{H-1/\alpha}M(ds)\\
=\int\limits_{\mathbb{R}}f(s-k)M(ds)
\end{align*}
and
$
\mid\mid\triangle_{k,1}X\mid\mid_\alpha^\alpha=\int\limits_{\mathbb{R}}\mid f(s-k)\mid^\alpha ds.
$
By taking the change of variable $u=s-k$, we get
\begin{align*}
 ||\triangle_{k,1}X||_\alpha^\alpha &%=\int_{\mathbb{R}}| f(s-k)|^\alpha d(s-k)=
 =\int\limits_{\mathbb{R}}| f(u)|^\alpha du
 =||\triangle_{0,1}X||_\alpha^\alpha.
 %&=C, {\text{  ($C$ is a constant)}}.\\
\end{align*}
Let 
$
U_k=\frac{\triangle_{k,1}X}{|| \triangle_{k,1}X||_\alpha}
$,
then
$
\mid\mid U_k\mid\mid_\alpha^\alpha=1
$
and
$
U_k=\int\limits_{\mathbb{R}}\frac{f(s-k)}{\mid\mid \triangle_{k,1}X\mid\mid_\alpha}M(ds).
 $
 We now prove that the assumption (\ref{eq.9.2}) is satisfied.
 Therefore, we consider 
\begin{align}\label{eq.50}
S_n=\frac{1}{n}\sum_{k\in \mathbb{Z},|k|\leq n}\mid cov(|\triangle_{k,1}X|^\beta,|\triangle_{0,1}X|^\beta)|.
\end{align}
Since $||\triangle_{k,1}X||_\alpha=||\triangle_{0,1}X||_\alpha$, it follows that
\[
 \sum_{k\in \mathbb{Z}, |k|\leq n}| cov(|\triangle_{k,1}X|^\beta,|\triangle_{0,1}X |^\beta)| =||\triangle_{0,1}X||_\alpha^{2\beta}
 \sum_{k\in \mathbb{Z},|k|\leq n}|cov(| U_k|^\beta,|U_0|^\beta)|.
\]
Moreover
\begin{align*}
[U_k,U_0]_2=\int\limits_{\mathbb{R}}\left|\frac{f(s-k)f(s)}{||\triangle_{0,1}X||_\alpha^2}\right|^{\alpha/2} ds
\end{align*}
Together with Lemma 3.6 in \cite{Jacques2012}, there exist $k_0>4K$ and $0<\eta<1$ such that for all $k\in\mathbb{Z}, |k|> k_0$,one has 
\begin{align*}
[U_k,U_0]_2\leq \eta<1.
\end{align*} 
Applying Theorem \ref{thm.2}, there exists $C(\eta)>0$ depending on $\eta$ such that
\[
|cov(| U_k|^\beta,| U_0|^\beta)|\leq C(\eta)\int\limits_{\mathbb{R}}| f(s-k)f(s)|^{\alpha/2}ds
\]
for all $|k|>k_0$. Then for $n>k_0$, one obtains that 
\begin{align*}
\sum_{k\in \mathbb{Z},|k|\leq n}|cov(| U_k|^\beta,|U_0|^\beta)|&=
\sum_{k\in \mathbb{Z}, |k| \leq k_0}|cov(| U_k|^\beta,| U_0|^\beta)|+
\sum_{k\in \mathbb{Z}, k_0<| k| \leq n}| cov(| U_k|^\beta,| U_0|^\beta)|\\
&\leq C\sum_{k\in \mathbb{Z}, | k| \leq k_0}\int\limits_{\mathbb{R}}|f(s-k)f(s)|^{\alpha/2}ds +
C \sum_{k\in \mathbb{Z}, k_0<|k| \leq n}|k|^{\frac{\alpha H-(L+1)\alpha}{2}}.
%&\leq \sum_{k\in \mathbb{Z}, \mid k\mid <4K}\mid cov(\mid U_k\mid^\beta,\mid U_0\mid^\beta)\mid+
%C \sum_{k\in \mathbb{Z}, |k|\leq n}|k|^{\frac{\alpha H-2\alpha}{2}}\\
\end{align*}
Because $f(x)\in L^\alpha(\mathbb {R},dx)$, one has
$
\sum\limits_{k\in \mathbb{Z}, |k|\leq k_0} \int\limits_{\mathbb{R}}\mid f(s-k)f(s)\mid^{\alpha/2}ds<+\infty
$.\\ 
Then 
$
S_n=\frac{C}{n} \sum\limits_{k\in \mathbb{Z}, k_0<|k| \leq n}|k|^{\frac{\alpha H-(L+1)\alpha}{2}}.
$\\
%We will use the following lemma to get the rate of convergence:
 Since $\alpha H-(L+1)\alpha <0$, using Lemma \ref{lem.B.1} in Appendix, we also get
$$
S_n=
\begin{cases}
 O(n^{-1})&\mbox{ if } H<L+1-\frac{2}{\alpha} \\ 
 O(n^{\frac{\alpha H-(L+1)\alpha}{2}})&\mbox{ if } H>L+1-\frac{2}{\alpha} \\ 
 O(\frac{\ln n}{n})&\mbox{ if } H=L+1-\frac{2}{\alpha}
\end{cases},
$$
where $S_n$ is defined by (\ref{eq.50}). Applying Theorem \ref{thm.5}, we have
$$
W_n(\beta)-\mathbb{E}|\triangle_{0,1}X|^{\beta}=
 O_{\mathbb P}(b_n), \widehat{H}_n -H=O_{\mathbb P}(b_n),
$$
where $b_n$ is defined by(\ref{eq.13}).
\end{proof}
\begin{proof}[\bf{Proof of Theorem \ref{thm.9}}]
We have
\begin{align*}
\triangle_{k,1}X & =\int\limits_0^{+\infty}\int\limits_{\mathbb{R}}
\left(\sum_{i=0}^Ka_i\mathbbm{1}_{S_{k+i}}(x,r)\right) M(dx,dr)\\
||\triangle_{k,1}X||_\alpha^\alpha & =\int\limits_0^{+\infty}\int\limits_{\mathbb{R}} 
|\sum_{i=0}^K a_i\mathbbm{1}_{S_{k+i}}(x,r)|^\alpha(r^{\nu-2})^\alpha dxdr\\
&=\int\limits_0^{+\infty}\int\limits_{\mathbb{R}} 
|\sum_{i=0}^K a_i\mathbbm{1}_{S_{i}}(x-k,r)|^\alpha(r^{\nu-2})^\alpha d(x-k)dr.
\end{align*}
By taking the change of variable $u=x-k$, one obtains that 
$$
||\triangle_{k,1}X||_\alpha^\alpha =\int\limits_0^{+\infty}\int\limits_{\mathbb{R}} 
|\sum_{i=0}^K a_i\mathbbm{1}_{S_{i}}(u,r)|^\alpha(r^{\nu-2})^\alpha dudr
=||\triangle_{0,1}X||_\alpha^\alpha.
$$
Set  
$
U_k=\frac{\triangle_{k,1}X}{||\triangle_{k,1}X||_\alpha}=\frac{\triangle_{k,1}X}{||\triangle_{0,1}X||_\alpha}.
$ Obviously, $||U_k||_\alpha^\alpha=1$.
We now prove that the condition (\ref{eq.9.2}) is satisfied.
Set 
$$
I_n=\sum_{k\in\mathbb{Z},|k|\leq n}
|cov(|\triangle_{k,1}X|^\beta,|\triangle_{0,1}X|^\beta)|=
||\triangle_{0,1}X||_\alpha^{\beta}\sum_{k\in\mathbb{Z},|k|\leq n}
|cov(|U_k|^\beta,|U_0|^\beta)|.
$$
For $n>2K$, applying Lemma \ref{lem.A.2} in Appendix, one gets
\begin{align*}
 I_n &\leq C\left(
 \sum_{k\in\mathbb{Z},|k|\leq 2K}
|cov(|U_k|^\beta,|U_0|^\beta)|+\sum_{k\in\mathbb{Z},2K<|k|\leq n}
|cov(|U_k|^\beta,|U_0|^\beta)|
 \right)\\
 &\leq C\left(
 \sum_{k\in\mathbb{Z},|k|\leq 2K}|cov(|U_k|^\beta,|U_0|^\beta)|+
 \sum_{k\in\mathbb{Z},2K<|k|\leq n}|k|^{\nu-1}\right).\\
\end{align*}
Since $0<\nu<1$, one gets $-1<\nu-1<0$. Following Lemma \ref{lem.B.1} in Appendix, one obtains 
$$
\frac{1}{n}\sum_{k\in\mathbb{Z},2K<|k|\leq n}|k|^{\nu-1}=O(n^{\nu-1}).
$$
Then we get the condition (\ref{eq.9.2}).
Applying Theorem \ref{thm.5}, we obtain that  
$
W_n(X)-\mathbb{E}|\triangle_{0,1}X|^{\beta}=O_{\mathbb{P}}(b_n)
$
and 
$$
\widehat{H}_n-H=O_{\mathbb{P}}(b_n), \hat{\alpha}_n-\alpha=O_{\mathbb{P}}(b_n)
$$
where $b_n$ is defined as in (\ref{eq.16}).
\end{proof}
\section{Appendix}\label{appendix}
We present here some technical results related to examples introduced in Section \ref{section.3}.
\subsection{Auxiliary results related to Fractional Brownian motion}
We are in position to provide and prove some technical results related to fractional Brownian motion. These results are used to present the variances for the limit distributions of the central limit theorems for the estimators of $H$ and $\alpha$ and to prove Theorem \ref{thm.6} in Subsection \ref{subsection.4.3}. 
\begin{prop}\label{prop.A.1}
Let $X$ be a $H$ fractional Brownian motion with $H\in (0,1)$. For $\beta\in\mathbb{R}, -1/2<\beta<0$, let $f_\beta$ be defined as in (\ref{eq.11.0}),
$$f_\beta=\sqrt{var \triangle_{0,1}X}^{\beta}(|x|^\beta-\mathbb{E}|Z_0|^\beta)$$
where $Z_0=\frac{\triangle_{0,1}X}{\sqrt{var \triangle_{0,1}X}}$. Then $f_\beta$ can be expanded in a unique way into series of Hermite polynomials
$$f_\beta(x)=\sum_{q \geq d}f_{\beta,q}H_q(x)$$
and $\sum\limits_{q\geq d}q!f^2_{\beta,q}<+\infty$, 
where $d$ is the Hermite rank of $f_\beta$, moreover $d\geq 2$.
\end{prop}
\begin{proof}
Since $-1/2<\beta<0$, one has
$\frac{1}{\sqrt{2\pi}}\int\limits_{\mathbb{R}}f_\beta(x)e^{-x^2/2}dx=0$ and 
$
 \frac{1}{\sqrt{2\pi}}\int\limits_{\mathbb{R}}f_\beta^2(x)e^{-x^2/2}dx<+\infty.
$
  Then following Proposition 1.4.2-(iv) in  \cite{Nourdin2012}, %from Property \ref{property.A.1}
we can write $f_\beta$ in terms of Hermite polynomials in a unique way %(see, e. g., Proposition 1.4.2-(iv), \cite{Nourdin2012})
\begin{align*}
f_\beta(x)=\sum_{q \geq d}f_{\beta,q}H_q(x),
\end{align*}
where $d\geq 1$ is the Hermite rank of $f_\beta$ and $H_q$s are the Hermite polynomials. \\
Moreover, it is clear that $Z_0\sim \mathcal{N}(0,1)$. From Proposition 2.2.1 in \cite{Nourdin2012}, we get 
$$
\mathbb{E}[H_p(Z_0)H_q(Z_0)]=
\begin{cases}
0& \textit{if } p\neq q\\
p!& \textit{if } p=q
\end{cases}
$$
 Then since $H_1(x)=x$, one has 
 $\mathbb{E}H_1(Z_0)f_\beta (Z_0)=\mathbb{E}Z_0f_\beta (Z_0)=f_{\beta,1}\mathbb{E}Z_0^2=f_{\beta,1}$. 
Combining with the fact that
\begin{align*}
\mathbb{E}H_1(Z_0)f_\beta (Z_0)&=\frac{(var \triangle_{0,1}X)^{\beta}}{\sqrt{2\pi}}\int\limits_{\mathbb{R}}
  x(|x|^\beta-\mathbb{E}|Z_0|^\beta)e^{-x^2/2}dx=0,
\end{align*}
we deduce that $f_{\beta,1}=0$. It follows that $d\geq 2$.\\
Moreover, 
\begin{align*}
\mathbb{E}f_\beta^2(Z_0)&=\frac{1}{\sqrt{2\pi}}\int\limits_{\mathbb{R}}f_\beta^2(x)e^{-x^2/2}dx<+\infty.
\end{align*}
On the other hand,  
\begin{align*}
 \mathbb{E}f_\beta^2(Z_0)&=\sum\limits_{p,q\geq d}f_{\beta,p}f_{\beta,q}\mathbb{E}[H_p(Z_0)H_q(Z_0)]=\sum\limits_{q\geq d}q!f^2_{\beta,q}. 
\end{align*}
It follows that 
$\sum\limits_{q\geq d}q!f^2_{\beta,q}<+\infty$.
%\begin{align*}
%f_{\beta,0}&=\frac{1}{\sqrt{2\pi}}\int\limits_{\mathbb{R}}f_\beta(x)e^{-x^2/2}dx=0.
%\end{align*}
%It follows that $d\geq 1$. 
 %From Proposition \ref{prop.A.4.2},
 \end{proof}
 \begin{lem}\label{lem.10}
 Let $(U,V)\stackrel{(d)}{=}\mathcal{N}_2\left((0,0),
 \begin{pmatrix}
                            1 & \rho\\
                            \rho & 1 
 \end{pmatrix}\right),\mid \rho\mid\leq 1$. Then for each $\beta\in\mathbb{C}, Re(\beta)\in(-1/2,0)$,  there exists a constant $C>0$ 
 such that 
  $\forall \mid\rho\mid\leq 1$, we have:
 $$\mid cov(\mid U\mid^\beta,\mid V \mid^\beta)\mid \leq C\rho^2$$
 \end{lem}
 \begin{proof}
 Let $\Sigma=\begin{pmatrix}
            1 & \rho\\
            \rho & 1
           \end{pmatrix}
$ . We have $det(\Sigma)=1-\rho^2$ and $\Sigma^{-1}=\frac{1}{1-\rho^2}\begin{pmatrix}
            1 & \rho\\
            \rho & 1
           \end{pmatrix}
$.
The density function of $(U,V)$:
\begin{align*}
 f(x,y)&=\mid 2\pi\Sigma\mid^{-1/2} exp\left[-\frac{1}{2}(x  y)\Sigma^{-1}\begin{pmatrix}
                                                                      x\\
                                                                      y
                                                                     \end{pmatrix}
\right]
=\left((2\pi)^2 det\Sigma\right)^{-1/2} exp\left[-\frac{1}{2(1-\rho^2)}(x^2+y^2-2\rho xy)\right].
\end{align*}
We get
\begin{align*}
 \mathbb{E}\left(\mid U\mid^\beta\mid V\mid^{\overline{\beta}}\right)&=\frac{1}{2\pi\sqrt{1-\rho^2}}
 \int\limits_{\mathbb{R}^2}\mid x\mid^\beta\mid y\mid^{\overline{\beta}}exp\left[-\frac{1}{2(1-\rho^2)}(x^2+y^2-2\rho xy)\right]dxdy\\
 \mathbb{E}\mid U\mid^\beta \mathbb{E}\mid V\mid^{\overline{\beta}}&=\frac{1}{2\pi}\int\limits_{\mathbb{R}^2}\mid x\mid^\beta
 \mid y\mid^{\overline{\beta}}exp(-\frac{x^2+y^2}{2})dxdy
\end{align*}
and
\begin{align*}
 cov(\mid U \mid^{\beta}, \mid V\mid^\beta)&=\mathbb{E}(\mid U\mid^\beta \mid V\mid^{\overline{\beta}})
 -\mathbb{E}\mid U\mid^\beta \mathbb{E}\mid V\mid^{\overline{\beta}}\\
 &=\frac{1}{2\pi}\int\limits_{\mathbb{R}^2}\mid x\mid^\beta\mid y\mid^{\overline{\beta}}exp(-\frac{x^2+y^2}{2}) A_{\rho}(x,y)dxdy
\end{align*}
where
$$
A_\rho(x,y)=\frac{1}{\sqrt{1-\rho^2}} exp\left( -\frac{\rho^2}{1-\rho^2}(x^2+y^2)\right) 
exp\left(\frac{\rho xy}{1-\rho^2}\right)-1.
$$
Since $\int\limits_{\mathbb{R}}\mid x\mid^\beta x e^{-x^2/2} dx=0$ we obtain that
$$
cov(\mid U \mid^{\beta}, \mid V\mid^\beta)=\frac{1}{2\pi}\int\limits_{\mathbb{R}^2}
\mid x\mid^\beta\mid y\mid^{\overline{\beta}}exp(-\frac{x^2+y^2}{2}) B_{\rho}(x,y)dxdy
$$
with
\begin{align*}
  B_{\rho}(x,y)&=A_\rho(x,y)-\rho xy=\frac{1}{\sqrt{1-\rho^2}}
 exp\left(-\frac{\rho^2}{1-\rho^2}(x^2+y^2)\right)
  exp\left(\frac{\rho xy}{1-\rho^2}\right)-1-\rho xy.
 \end{align*}
 Using L'Hôpital rule, we get:
 \begin{align*}
  \lim_{\rho\to 0}\frac{B_{\rho}}{\rho^2}&=\lim_{\rho\to 0}\frac{B_{\rho}^{'}}{2\rho}\\
  B_\rho^{'}&=A_\rho^{'}(x,y)-xy\\
  &=\left[\left( \rho (1-\rho^2)^\frac{-3}{2}-(x^2+y^2)[2\rho(1-\rho^2)^{-1}+2\rho^3(1-\rho^2)^{-2}]\right)
  +xy\left( (1-\rho^2)^{-1}+2\rho^2(1-\rho^2)^{-2}  \right)\right]\\
  &\times exp\left(-\frac{\rho^2}{1-\rho^2}(x^2+y^2)\right)
  exp\left(\frac{\rho xy}{1-\rho^2}\right)-xy \\
\lim_{\rho\to 0}\frac{B_{\rho}^{'}}{2\rho}&=xy\cdot\lim_{\rho\to 0}
\frac{\left[ exp\left(\frac{-\rho^2}{1-\rho^2}(x^2+y^2)+
\frac{\rho xy}{1-\rho^2}\right)(1-\rho^2)^{-1}-1\right]}{2\rho}+\frac{1}{2}-(x^2+y^2):=A+\frac{1}{2}-(x^2+y^2).
\end{align*}
Then we continue using L'Hôpital rule for the remaining limit:
\begin{align*}
A=&\frac{xy}{2}\lim_{\rho\to 0}exp\left(\frac{-\rho^2}{1-\rho^2}(x^2+y^2)+
\frac{\rho xy}{1-\rho^2}\right) \times\\
&\left[ \frac{2\rho}{(1-\rho^2)^2}-(x^2+y^2)\left(\frac{2\rho}{1-\rho^2}
+\frac{2\rho^3}{(1-\rho^2)^2}\right)+xy\left( \frac{1}{1-\rho^2}+\frac{2\rho^2}{(1-\rho)^2}\right)\right]\\
=&\frac{x^2y^2}{2}.
\end{align*}
One has
\begin{align*}
\lim_{\rho\to 0}\frac{B_\rho}{\rho^2}&=\frac{x^2y^2+1}{2}-(x^2+y^2)\\
\frac{\partial^2B_\rho(x,y)}{\partial\rho^2}&=P_\rho(x,y)exp\left(-\frac{\rho^2}{1-\rho^2}(x^2+y^2)\right)
  exp\left(\frac{\rho xy}{1-\rho^2}\right).
\end{align*}
 where $P_\rho(x,y)$ is a fourth degree polynomial that depends continuously on $\rho$. We also have 
 $$B_0(x,y)=0, B_\rho^{'}(x,y)\mid_{\rho=0}=0.$$
 A Taylor expansion up to order 2 leads to
 $$
 B_\rho(x,y)=\rho^2P_{\tilde{\rho}}(x,y)exp\left(-\frac{{\tilde{\rho}}^2}{1-{\tilde{\rho}}^2}(x^2+y^2)\right)
  exp\left(\frac{\tilde{\rho} xy}{1-{\tilde{\rho}}^2}\right)
 $$
 with $\tilde{\rho}\in(0,\rho)$.
 On the compact set $\mid \rho\mid\leq 1/2$, the polynomial $P_\rho(x,y)$ can be bounded by a fourth degree polynomial
 $P(x,y)$, for all $x,y\in\mathbb{R}, \mid P_\rho(x,y)\mid\leq\mid P(\mid x\mid, \mid y\mid)\mid$. Moreover
  $$
 exp\left(-\frac{{\tilde{\rho}}^2}{1-{\tilde{\rho}}^2}(x^2+y^2)\right)
  exp\left(\frac{\tilde{\rho} xy}{1-{\tilde{\rho}}^2}\right)\leq 
  exp\left(\frac{\tilde{\rho} xy}{1-{\tilde{\rho}}^2}\right).
 $$
 But with $\mid \tilde{\rho}\mid\leq 1/2$, we get $\mid \frac{\tilde{\rho}}{1-{\tilde{\rho}}^2}\mid\leq 2/3$. So
 $ exp\left(\frac{\tilde{\rho} xy}{1-{\tilde{\rho}}^2}\right)\leq exp \left(\frac{2\mid x y\mid}{3}\right)$. Thus
 $$\mid B_\rho (x,y)\mid\leq \rho^2\mid P(\mid x\mid,\mid y\mid)\mid exp \left(\frac{2\mid x y\mid}{3}\right).$$
 Because the power function grows faster than the polynomial function, we have
 $$
 \int\limits_{\mathbb{R}^2}\mid xy\mid ^{Re(\beta)}exp(-\frac{x^2+y^2}{2})P(\mid x\mid,\mid y\mid)\mid exp \left(\frac{2\mid x y\mid}{3}\right)dxdy<\infty
 .$$
 So we have the conclusion.
  \end{proof}
 \begin{lem}\label{lem.11}
  Let $X$ be a fractional Brownian motion, $\beta\in\mathbb{C},Re(\beta)\in(-1/2,0)$. Then
  $$\sum_{k\in \mathbb{Z}}\mid cov(\mid \triangle_{k,1}X\mid^\beta,\mid \triangle_{0,1}X\mid^\beta)\mid <+\infty.$$
\end{lem}
\begin{proof}
We have
 \begin{align*}
 cov(\triangle_{k,1}X,\triangle_{0,1}X)&=-\frac{1}{2}\sum_{p,p'=0}^{K}a_pa_{p'}\mid k+p-p'\mid^{2H}=-\frac{k^{2H}}{2}\sum_{p,p'=0}^Ka_pa_{p'}\mid 1+\frac{p-p'}{k}\mid^{2H}.
 \end{align*}
 We just need to consider $cov(\triangle_{k,1}X,\triangle_{0,1}X)$ when $\mid k\mid\geq K$. Since 
$ 1+\frac{p-p'}{k}\geq 0$, we get
 $$cov(\triangle_{k,1}X,\triangle_{0,1}X)=-\frac{k^{2H}}{2}\sum_{p,p'=0}^Ka_pa_{p'}(1+\frac{p-p'}{k})^{2H}.$$
 Set
 $$g(x)=-\frac{1}{2}\sum_{p,p'=0}^Ka_pa_{p'}(1+(p-p')x)^{2H}.$$
 If $H=\frac{1}{2}$ then $g(x)=0$.
 If $H\neq\frac{1}{2}$, using Taylor expansion as $x\rightarrow 0$, we get
 \begin{align*}
  g'(x)&=-H\sum_{p,p'=0}^Ka_pa_{p'}(p-p')(1+(p-p')x)^{2H-1}\\
  g''(x)&=-H(2H-1)\sum_{p,p'=0}^Ka_pa_{p'}(p-p')^2(1+(p-p')x)^{2H-2}\\
  g^{(3)}(x)&=-H(2H-1)(2H-2)\sum_{p,p'=0}^Ka_pa_{p'}(p-p')^3(1+(p-p')x)^{2H-3}.
 \end{align*}
 Thus
 $g(0)=0, g'(0)=0,g''(0)=0,g^{(3)}(0)=0$ and we obtain that $g(x)=o(x^3)$ as $x\rightarrow 0$. It follows that
 $$cov(\triangle_{k,1}X,\triangle_{0,1}X)\sim k^{2H}\cdot o(\frac{1}{k^3})\sim o(k^{2H-3})$$
 as $k\rightarrow +\infty$.
 We can apply similarly as $k\rightarrow -\infty$. Then there exists a constant $C$ such that for all
 $ k, \mid k\mid \geq K$ and for all $H\in (0,1),$
 \begin{equation}\label{eq.12}
 \mid cov(\triangle_{k,1}X,\triangle_{0,1}X)\mid \leq C \mid k\mid^{2H-3}.
 \end{equation}
For all $k\in \mathbb {Z}$, we have
\begin{align*}
 var \triangle_{k,1}X&=\mathbb {E}\left[\sum_{p=0}^Ka_pX(k+p)\sum_{p'=0}^Ka_{p'}X(k+p')\right]
 %&=\sum_{p,p'=0}^Ka_pa_{p'}\mathbb{E}X(k+p)X(k+p')\\
 %&=-\frac{1}{2}\sum_{p,p'=0}^Ka_pa_{p'}(\mid k+p\mid^{2H}+\mid k+p'\mid^{2H}-\mid p-p'\mid^{2H}\\
 =-\frac{1}{2}\sum_{p,p'=0}^Ka_pa_{p'}\mid p-p'\mid^{2H}.
\end{align*}
We now apply the Lemma \ref{lem.10} with $U=\frac{\triangle_{k,1}X}{\sqrt{var(\triangle_{k,1}X)}}, 
V=\frac{\triangle_{0,1}X}{\sqrt{var(\triangle_{0,1}X)}}$. Then 
\begin{align*}
 \left|cov\left(\mid \frac{\triangle_{k,1}X}{\sqrt{var(\triangle_{k,1}X)})}\mid^\beta,
 \mid \frac{\triangle_{0,1}X}{\sqrt{var(\triangle_{0,1}X)})}\mid^\beta\right)\right|
 \leq C\cdot \frac{{cov}^2(\triangle_{k,1}X,\triangle_{0,1}X)}{{var}^2 \triangle_{0,1}X}.
\end{align*}
It follows that
$$\left| cov(\mid \triangle_{k,1}X\mid^\beta,\mid \triangle_{0,1}X\mid^\beta)\right|
\leq C{cov}^2(\triangle_{k,1}X,\triangle_{0,1}X), \forall k, k\in\mathbb{Z}.$$%\leq C_3 \mid k\mid^{4H-6}$$
Since $H\in (0,1)$, we get $\sum\limits_{k\in \mathbb{Z}}\mid k\mid^{4H-6}<+\infty.$
Applying inequality (\ref{eq.12}), we obtain 
\begin{align*}
\sum_{k\in \mathbb{Z}}\left|  cov(\mid \triangle_{k,1}X\mid^\beta,\mid \triangle_{0,1}X\mid^\beta)\right| 
=&\sum_{k\in \mathbb{Z}, \mid k\mid <K}\left| cov(\mid \triangle_{k,1}X\mid^\beta,\mid \triangle_{0,1}X\mid^\beta)\right|\\
&+
\sum_{k\in \mathbb{Z},\mid k\mid \geq K}\left| cov(\mid \triangle_{k,1}X\mid^\beta,\mid \triangle_{0,1}X\mid^\beta)\right|\\
\leq& C\sum_{k\in \mathbb{Z}, \mid k\mid<K}{cov}^2(\triangle_{k,1}X,\triangle_{0,1}X)
+C\sum_{k\in \mathbb{Z},\mid k\mid \geq K}\mid k\mid^{4H-6}<+\infty.
\end{align*}
\end{proof}
\subsection{Auxiliary results related to Takenaka's process}
\begin{lem}\label{lem.A.2}
 Let $\{X_t, t\in\mathbb{R}\}$ be a Takenaka's process defined by (\ref{eq.15}).
  Then for $\beta \in \mathbb{R}, \beta\in(-1/2,0)$ and $|k|>2K$, we have
$$
\left| cov(\mid \triangle_{k,1}X \mid^{\beta}, \mid \triangle_{0,1}X \mid^{\beta})\right|\leq Ck^{\nu-1}.
$$
\end{lem}
\begin{proof}
%We first need to estimate $cov(\mid \triangle_{k,1} X\mid^{\beta},\mid \triangle_{0,1}X \mid^{\beta})$.\\
One has
$$
\triangle_{k,1}X=\sum_{i=0}^K a_iX(k+i)=
\int\limits_{\mathbb{R}\times\mathbb{R}^+}\sum_{i=0}^Ka_i\mathbbm{1}_{S_{k+i}}(x,r)M(dx,dr)
$$
where
$
f_k=\sum\limits_{i=0}^K a_i\mathbbm{1}_{S_{k+i}}=\sum\limits_{i=0}^K a_i(\mathbbm{1}_{C_{k+i}}-\mathbbm{1}_{C_0})^2.
$
From the fact that $\sum\limits_{i=0}^K a_i=0$ and $\mid 1-2\mathbbm{1}_{C_0}\mid=1$, it induces
\begin{align*}
\mid f_k\mid &=\mid 1-2\mathbbm{1}_{C_0}\mid \mid\sum_{i=0}^K a_i\mathbbm{1}_{C_{k+i}}\mid=\mid\sum_{i=0}^K a_i\mathbbm{1}_{C_{k+i}}\mid.
\end{align*}
Therefore we have to estimate, as $|k|\rightarrow +\infty$,
$$
I_k=[\triangle_{k,1}X, \triangle_{0,1}X]_2=\int\limits_0^{+\infty}r^{\nu-2}\int\limits_{\mathbb{R}}\mid f_k(x,r)f_0(x,r)\mid dx dr.
$$
We will find an upper bound for $I_k$ when $|k|\geq 2K$.\\
If $x>K+r$ then $\mathbbm{1}_{C_{i}}(x,r)=0$ for all $i=0,\ldots,K$, thus $f_0(x,r)=0$.
If $x<k-r$, $\mathbbm{1}_{C_{k+i}}(x,r)=0$ for all $i=0,\ldots,K$, it follows that $f_k(x,r)=0$. \\
As a result,  $f_k(x,r)f_0(x,r)=0$ for all $x\in (-\infty,k-r)\cup (K+r,+\infty)$.\\
Let $k>2K$. If
$r<\frac{k-K}{2}\Leftrightarrow k-r>K+r$ then $f_k(x,r)f_0(x,r)=0$ for all $x$. Thus one gets
\begin{align*}
 I_k &=\int\limits_{\mathbb{R}}\left(\int\limits_0^{\frac{k-K}{2}}r^{\nu-2}\mid f_k(x,r)f_0(x,r)\mid dr
 +\int\limits_{\frac{k-K}{2}}^{+\infty}r^{\nu-2}\mid f_k(x,r)f_0(x,r)\mid dr\right) dx\\
 &=\int\limits_{\mathbb{R}}\int\limits_{\frac{k-K}{2}}^{+\infty}r^{\nu-2}\mid f_k(x,r)f_0(x,r)\mid dr dx=\int\limits_{\frac{k-K}{2}}^{+\infty}r^{\nu-2}\int\limits_{\mathbb{R}}\mid f_k(x,r)f_0(x,r)\mid dx dr.
\end{align*}
Here we consider $f_k(x,r)f_0(x,r)$.\\
Since $k>2K$, then $k+K-r\leq k+r$. For $k+K-r\leq x\leq  k+r$, $\mid x-k-i\mid\leq r$
for all $ i=0,\ldots,K$, it follows that $\mathbbm{1}_{C_{k+i}}(x,r)=1$ and $f_k(x,r)=\sum\limits_{i=0}^Ka_i=0$.\\
Therefore $f_k(x,r)f_0(x,r)=0$ if $x\in(-\infty, k-r)\cup (k+K-r, +\infty)$.
We also have
$$
\mid f_k(x,r)\mid=\mid \sum_{i=0}^Ka_i\mathbbm{1}_{C_{k+i}}(x,r)\mid \leq \sum_{i=0}^K\mid a_i\mid
$$
 for all $k\in\mathbb{N}$.
 Thus
\begin{align*}
 I_k&\leq \int\limits_{\frac{k-K}{2}}^{+\infty}r^{\nu-2}\int\limits_{k-r}^{k+K-r}\mid f_k(x,r)f_0(x,r)\mid dxdr\\
 &\leq C\int\limits_{\frac{k-K}{2}}^{+\infty}r^{\nu-2}\int\limits_{k-r}^{k+K-r} dxdr=C(\frac{k-K}{2})^{\nu-1}\leq Ck^{\nu-1}
\end{align*}
since $\frac{k-K}{2}\geq \frac{k}{4}$ and $0<\nu<1$. \\
Let $k<-2K$. If $r<-\frac{k+K}{2}\Leftrightarrow k+K+r<-r$, then for all $i=1,\ldots,K$,
$$
\mathbbm{1}_{C_{k+i}}(x,r)=0, \forall x\in(k+K+r,+\infty),
\mathbbm{1}_{C_i}(x,r)=0, \forall x\in(-\infty,-r).
$$
It follows that
$f_k(x,r)f_0(x,r)=0$, for all $x\in (-\infty,-r)\cup (k+K+r,+\infty)=\mathbb{R}$. Therefore
$$
I_k=\int\limits_{-\frac{k+K}{2}}^{+\infty}r^{\nu-2}\int\limits_{\mathbb{R}}|f_k(x,r)f_0(x,r)|dxdr
$$
For $r>-\frac{k+K}{2}$, $r>K-K/2=K/2$. We have $f_k(x,r)f_0(x,r)=0$ for all $x\in (-\infty,k-r)\cup(k-r+K,k+r) \cup (k+r+K,K+r)\cup (K+r,+\infty)$. 
It induces that  
\begin{align*}
I_k&=\int\limits_{-\frac{k+K}{2}}^{+\infty}r^{\nu-2}\int\limits_{\mathbb{R}}|f_k(x,r)f_0(x,r)|dxdr\\
&=\int\limits_{-\frac{k+K}{2}}^{+\infty}r^{\nu-2}
\left(\int\limits_{k-r}^{k-r+K}|f_k(x,r)f_0(x,r)|dxdr+
\int\limits_{k+r}^{k+r+K}|f_k(x,r)f_0(x,r)|dxdr \right)\\
&\leq C\int\limits_{-\frac{k+K}{2}}^{+\infty}r^{\nu-2}dr\leq C|k|^{\nu-1}.
\end{align*}
Putting together with Theorem \ref{thm.2}, for $|k|>2K$ we obtain that 
$$
\left| cov(\mid \triangle_{k,1}X \mid^{\beta}, \mid \triangle_{0,1}X \mid^{\beta})\right|\leq Ck^{\nu-1}.
$$
\end{proof}
\subsection{Auxiliary results related to rate of convergence}
We present here a lemma used to determine rate of convergence in the proofs of Theorems 3.3 and 3.4.
\begin{lem}\label{lem.B.1}
For $p<0$, let $
S_n=\frac{1}{n}\sum_{|k|\leq n}|k|^p
$, then 
$
\lim\limits_{n\rightarrow +\infty}S_n=0.
$ Moreover
$$
S_n=
\begin{cases} 
O(n^{-1}) &\mbox{ if } p<-1\\
O(n^{p}) &\mbox{ if } -1<p<0\\
 O(\frac{\ln n}{n}) &\mbox{ if } p=-1.
\end{cases}
$$
\end{lem}
\begin{proof}
Set
\[
S_n=\frac{1}{n}\sum_{k=1}^nk^p.
\]
If $p<-1$, since 
$
\int\limits_{1}^\infty x^pdx<+\infty,$
 following the integral test for convergence, we get 
 $
 \sum\limits_{k=1}^{\infty} k^p<+\infty.
 $ Then 
 $$
 S_n=O(n^{-1}).%, \lim_{n\rightarrow +\infty} S_n=0.
$$
If $-1<p< 0$, we take a constant $\epsilon $ such that $0<\epsilon<-p$, then
\begin{align*}
S_n&=\frac{1}{n}\sum\limits_{k=1}^n\frac{k^{1+\epsilon+p}}{k^{1+\epsilon}}
\leq \frac{1}{n}\sum\limits_{k=1}^n \frac{n^{1+\epsilon+p}}{k^{1+\epsilon}}
% {\text{  (because $1+\epsilon+p>0$, so $
%k^{1+\epsilon+p}\leq n^{1+\epsilon+p}
%$)  }} \\
=n^{p+\epsilon}\sum\limits_{k=1}^n\frac{1}{k^{1+\epsilon}}.
\end{align*}
 Since $p+\epsilon<0$, we get
$\sum\limits_{k=1}^n\frac{1}{k^{1+\epsilon}}<+\infty $. Then 
$S_n=O(n^{p+\epsilon})$
for all $0<\epsilon<-p$. Thus 
$S_n=O(n^{p})$. \\
If $p=-1$, then 
$
S_n =\frac{1}{n}\sum\limits_{k=1}^n\frac{1}{k}=O(\frac{\ln n}{n})
%&\leq \frac{1}{n}\sum_{k=1}^n \frac{n^{\epsilon}}{k^{1+\epsilon}}\\
%&=n^{\epsilon-1}}\sum_{k=1}^n\frac{1}{k^{1+\epsilon}}.
$.\\
In all cases, we have
$\lim\limits_{n\rightarrow +\infty} S_n=0$.
\end{proof}
%\subsection{Auxiliary results related to well-balanced linear fractional stable motion}
\section*{Acknowledgements}
The authors would like to thank Michel Zinsmeister for his helpful suggestions. 

 \bibliographystyle{plain}
\bibliography{references}
\end{document}